\documentclass[12pt]{amsart}
\usepackage{amsmath,amscd,amsthm,verbatim,alltt,amsfonts,array}
\usepackage[english]{babel}
\usepackage{latexsym}
\usepackage{amssymb}
\usepackage{euscript}
\usepackage{nicefrac}
\usepackage{graphicx}
\usepackage[all]{xy}

\newtheorem{Theorem}{Theorem}
\theoremstyle{plain}

\newtheorem{Conjecture}{Conjecture}

\newtheorem{Example}{Example}

\newtheorem{Lemma}{Lemma}

\newtheorem{Proposition}{Proposition}
\newtheorem{Remark}{Remark}

\numberwithin{equation}{section}
\let\epsilon\varepsilon
\let\phi=\varphi
\let\kappa=\varkappa
\def\opn#1#2{\def#1{\operatorname{#2}}} % to make operators
\opn\sdepth{sdepth}
\opn\depth{depth}
\opn\chara{char} \opn\length{\ell} \opn\pd{pd} \opn\rk{rk}
\opn\projdim{proj\,dim} \opn\injdim{inj\,dim} \opn\rank{rank}
\opn\depth{depth} \opn\sdepth{sdepth} \opn\fdepth{fdepth}
\opn\grade{grade} \opn\height{height} \opn\embdim{emb\,dim}
\opn\codim{codim}  \opn\min{min} \opn\max{max}

\opn\Tr{Tr} \opn\bigrank{big\,rank}
\opn\superheight{superheight}\opn\lcm{lcm}
\opn\trdeg{tr\,deg}
\opn\reg{reg} \opn\lreg{lreg} \opn\ini{in} \opn\lpd{lpd}
\opn\size{size}

\opn\Spec{Spec} \opn\Supp{Supp} \opn\supp{supp} \opn\Sing{Sing}
\opn\Ass{Ass} \opn\Min{Min}
\opn\Ann{Ann} \opn\Rad{Rad} \opn\Soc{Soc}
\opn\Im{Im} \opn\Ker{Ker} \opn\Coker{Coker} \opn\Am{Am}
\opn\Hom{Hom} \opn\Tor{Tor} \opn\Ext{Ext} \opn\End{End}
\opn\Aut{Aut} \opn\id{id}  \opn\deg{deg}

\opn\nat{nat}
\opn\pff{pf}%   \pf exists already
\opn\Pf{Pf} \opn\GL{GL} \opn\SL{SL} \opn\mod{mod} \opn\ord{ord}
\opn\Gin{Gin} \opn\Hilb{Hilb}
\def\qed{\ifhmode\textqed\fi
      \ifmmode\ifinner\quad\qedsymbol\else\dispqed\fi\fi}
\def\textqed{\unskip\nobreak\penalty50
       \hskip2em\hbox{}\nobreak\hfil\qedsymbol
       \parfillskip=0pt \finalhyphendemerits=0}
\def\dispqed{\rlap{\qquad\qedsymbol}}

\begin{document}
\title{\bf Stanley depth on five generated, squarefree, monomial ideals  }

\author{ Dorin Popescu}

\thanks{The  support from  grant ID-PCE-2011-1023 of Romanian Ministry of Education, Research and Innovation is gratefully acknowledged.}

\maketitle
\begin{abstract} Let $I\supsetneq J$ be  two  squarefree monomial ideals of a polynomial algebra over a field generated in degree $\geq d$, resp. $\geq d+1$ .  Suppose that  $I$ is either generated  by four squarefree monomials of degrees $d$ and others of degrees $\geq d+1$, or by five special monomials of degrees $d$. If  the Stanley depth of $I/J$ is $\leq d+1$ then the usual depth of $I/J$ is $\leq d+1$ too.

 \noindent
  {\it Key words } : Monomial Ideals,  Depth, Stanley depth.\\
 {\it 2010 Mathematics Subject Classification: Primary 13C15, Secondary 13F20, 13F55,
13P10.}
\end{abstract}

\section*{Introduction}
Let $K$ be a field and $S=K[x_1,\ldots,x_n]$ be the polynomial $K$-algebra in $n$ variables. Let  $I\supsetneq J$ be  two   squarefree monomial ideals of $S$ and suppose that  $I$ is generated by squarefree monomials of degrees $\geq d$   for some positive integer $d$.  After  a multigraded isomorphism we may assume either that $J=0$, or $J$ is generated in degrees $\geq d+1$.

 Let $P_{I\setminus J}$  be the poset of all squarefree monomials of $I\setminus J$  with the order given by the divisibility. Let $ P$ be a partition of  $P_{I\setminus J}$ in intervals $[u,v]=\{w\in  P_{I\setminus J}: u|w, w|v\}$, let us say   $P_{I\setminus J}=\cup_i [u_i,v_i]$, the union being disjoint.
Define $\sdepth  P=\min_i\deg v_i$ and  the  {\em Stanley depth} of $I/J$ given by $\sdepth_SI/J=\max_{ P} \sdepth  P$, where $ P$ runs in the set of all partitions of $P_{I\setminus J}$ (see  \cite{HVZ}, \cite{S}).
 Stanley's Conjecture says that  $\sdepth_S I/J\geq \depth_S I/J$.

 In spite of so many papers on this subject (see \cite{HVZ}, \cite{P2}, \cite{R}, \cite{BKU}, \cite{HPV}, \cite{Sh}, \cite{P0}, \cite{Is}, \cite{Ci}, \cite{P}, \cite{PZ2}) Stanley's Conjecture remains open after more than thirty years. Meanwhile, new concepts as for example the Hilbert depth (see \cite{BKU}, \cite{U}, \cite{IM}) proved to be helpful in this area (see for instance \cite[Theorem 2.4]{Sh}). Using a Theorem of Uliczka \cite{U} it was shown in \cite{AP1}  that for $n=6$ the Hilbert depth of $S\oplus m$ is strictly bigger than the Hilbert depth of $m$, where $m$ is the maximal graded ideal of $S$. Thus  for $n=6$ one could also expect
 $\sdepth_S(S\oplus m)>\sdepth_Sm$, that is a negative answer for a Herzog's question.  This was stated later by Ichim and Zarojanu \cite{IZ}.

 Suppose that $I \subset S$ is minimally generated by some squarefree monomials $f_1,\ldots,f_r$ of degrees $d$,  and a  set $E$  of squarefree monomials of degree $\geq d+1$. By \cite[Proposition 3.1]{HVZ} (see  \cite[Lemma 1.1]{P}) we have $\depth_S I/J\geq d$. Thus if $\sdepth_S I/J=d$ then Stanley's Conjecture says that  $\depth_S I/J = d$. This is exactly what
  \cite[Theorem 4.3]{P}) states.
   Next step in studying Stanley's Conjecture is to prove the following weaker one.

\begin{Conjecture} \label{c}   Suppose that $I \subset S$ is minimally generated by some squarefree monomials $f_1,\ldots,f_r$ of degrees $d$,  and a  set $E$  of squarefree monomials of degree $\geq d+1$. If  $\sdepth_S I/J=d+1$ then $\depth_S I/J \leq d+1$.
\end{Conjecture}
This conjecture is studied in \cite{PZ}, \cite{PZ1}, \cite{PZ2} either when $r=1$, or when $E=\emptyset$ and  $r\leq 3$. Recently, these results were improved in the next theorem.

\begin{Theorem}\label{ap} (A. Popescu, D.Popescu \cite[Theorem 0.6]{AP})  Let   $C$ be the set of the squarefree monomials of degree  $d+2$ of $I\setminus J$. Conjecture \ref{c} holds  in each of the following two  cases:
\begin{enumerate}
\item{} $r\leq 3$,
\item{} $r=4$, $E=\emptyset$ and there exists $c\in C$ such that $\supp c\not \subset \cup_{i\in [4]} \supp f_i$.
\end{enumerate}
\end{Theorem}

 The  purpose of this paper is to extend the above theorem in the following form.

\begin{Theorem}\label{m}   Let   $B$ be the set of the squarefree monomials of degree  $d+1$ of $I\setminus J$.  Conjecture \ref{c} holds  in each of the following two  cases:
\begin{enumerate}
\item{} $r\leq 4$,
\item{} $r=5$,  and there exists  $t\not\in \cup_{i\in [5]} \supp f_i$, $t\in [n]$ such that $(B\setminus E)\cap (x_t)\not =\emptyset$ and $E\subset (x_t)$.
\end{enumerate}
\end{Theorem}
The above theorem follows from  Theorems \ref{m1}, \ref{m2} (the case $r=4$, $E=\emptyset$ is given already in Proposition \ref{m'}). It is worth to mention that the idea of the proof of Proposition \ref{m'}, and  Theorem \ref{ap} started already in the proof of \cite[Lemma 4.1]{PZ2} when $r=1$. Here  {\em path} is a more general notion,  the reason being to suit better the exposition. However, the case $r=4$, $E\not=\emptyset$ is more complicated (see Remark \ref{M''}) and we have to study separately  the special case when $f_i\in (v)$, $i\in [4]$ for some monomial $v$ of degree $d-1$ (see the proof of Theorem \ref{m1}).

What can be done next? We believe that Conjecture \ref{c} holds, but the proofs will become harder with  increasing $r$. Perhaps for each $r\geq 5$ the proof could be done in  more or less a common form but leaving some ''pathological'' cases which should be done separately. Thus to get a proof of Conjecture \ref{c} seems to be a difficult aim.

We owe thanks to a Referee, who noticed  some  mistakes in a previous version of this paper, especially in the proof of Lemma \ref{dep}.

\section{Depth and Stanley depth}

Suppose that $I$ is minimally generated by some squarefree monomials $f_1,\ldots,f_r$ of degrees $ d$  for some $d\in {\mathbb N}$ and a set of squarefree monomials $E$ of degree $\geq d+1$.
 Let  $B$ (resp. $C$) be the set of the squarefree monomials of degrees $d+1$  (resp. $d+2$) of $I\setminus J$. Set $s=|B|$, $q=|C|$.  Let $w_{ij}$ be the least common multiple of $f_i$ and $f_j$ and set $W$ to be the set of all $w_{ij}$.  Let $C_3$ be the set of all $c\in C\cap  (f_1,\ldots,f_r)$ having all divisors from $B\setminus E$ in $W$. In particular each monomial of $C_3 $ is the least common multiple of three of the $f_i$. The converse is not true as shown by  \cite[Example 1.6]{AP}. Let $C_2$ be the set of all $c\in C$, which are the least common multiple of two $f_i$, that is $C_2=C\cap W$. Then $C_{23}=C_2\cup C_3$ is the set of all $c\in C$, which are the least common multiple of two or three $f_i$. We may have $C_2\cap C_3\not = \emptyset$ as shows the following example.

 \begin{Example} \label{exe}{\em Let $n\geq 4$, $f_i=x_ix_{i+1}$, $i\in [3]$, $f_4=x_1x_4$ and $I=(f_1,\ldots,f_4)$, $J=0$. Note that $m=x_1x_2x_3x_4$ is a least common multiple of every three monomials $f_j$ and the divisors of $m$ with degree $3$ are $w_{12}, w_{23},w_{34}, w_{14}$. Thus $m\in C_3$. But $m\in C_2$ because $m=w_{13}=w_{24}$.}
 \end{Example}

We start with a  lemma, which slightly extends \cite[Theorem 2.1]{AP}.
\begin{Lemma} \label{el} Suppose that
there exists  $t\in [n]$, $t\not\in  \cup_{i\in [r]} \supp f_i$ such that $(B\setminus E)\cap (x_t)\not =\emptyset$ and $E\subset (x_t)$.
If Conjecture \ref{c} holds for  $r'<r$  and
 $\sdepth_SI/J=d+1$,
then $\depth_SI/J\leq d+1$.
\end{Lemma}
\begin{proof}  We follow the proof of \cite[Theorem 2.1]{AP}. Apply induction on $|E|$, the case $|E|=0$ being done in the quoted theorem. We may suppose
 that  $E$ contains only monomials of degrees $d+1$ by \cite[Lemma 1.6]{PZ}. Since Conjecture \ref{c} holds for $r'<r$ we see that $C\not\subset (f_2,\ldots,f_r,E)$ implies $\depth_S I/J\leq d+1$ by  \cite[Lemma 1.1]{PZ2}. If Conjecture \ref{c} holds for $r$ and $E\setminus\{a\}$ with some $a\in E$ then  $C\not\subset (f_1,\ldots,f_r,E\setminus \{a\})$ implies again $\depth_S I/J\leq d+1$ by  the quoted lemma. Thus using the induction hypothesis on $|E|$
 we may assume that $C\subset (W)\cup ((E)\cap (f_1,\ldots,f_r))\cup (\cup_{a,a'\in E, a\not =a'} (a)\cap (a'))$.
  Let $I_t=I\cap (x_t)$, $J_t=J\cap (x_t)$, $B_t=(B\setminus E)\cap (x_t)=\{x_tf_1,\ldots,x_tf_e\}$,
 for some $1\leq e\leq r$. If $\sdepth_SI_t/J_t\leq d+1$ then $\depth_SI_t/J_t\leq d+1$ by \cite[Theorem 4.3]{P}  because $I_t$ is generated only by monomials of degree $d+1$. Thus $\depth_SI/J\leq \depth_SI_t/J_t\leq d+1$ by \cite[Lemma 1.1]{AP}.

 Suppose that $\sdepth_SI_t/J_t\geq d+2$. Then there exists a partition on $I_t/J_t$ with sdepth $d+2$ having some disjoint intervals $[x_tf_i,c_i]$, $i\in [e]$ and $[a,c_a]$, $a\in E$. We may assume that $c_i,c_a$ have degrees $d+2$.
  We have either $c_i\in (W)$, or $c_i\in ((E)\cap (f_1,\ldots,f_r))\setminus(W)$. In the first case   $c_i=x_tw_{ik_i}$ for some $1\leq k_i\leq r$, $k_i\not = i$. Note that $x_tf_{k_i}\in B$ and so  $k_i\leq e$. We consider
the intervals $[f_i,c_i]$. These intervals contain $x_tf_i$ and possible a  $w_{ik_i}$.  If $w_{ik_i} =w_{jk_j}$  for $i\not =j$ then we get $c_i=c_j$ which is false. Thus these intervals are disjoint.

   Let $I'$ be the ideal generated by  $f_j$ for $ e<j\leq r$ and
  $B\setminus (E\cup (\cup_{i=1}^ e[f_i, c_i]))$. Set $J'=I'\cap J$. Note that $I'\not =I$ because $e\geq 1$ . As we showed already $c_i\not \in I'$ for any $i\in [ e]$. Also $c_a\not\in I'$  because otherwise $c_a=x_tx_kf_j$ for some $e<j\leq r$ and we get $x_tf_j\in B$, which is false.
 In the following exact sequence
$$0\to I'/J'\to I/J\to I/(J+I')\to 0$$
the last term has a partition of sdepth $d+2$ given by the intervals  $[f_i,c_i]$ for $1\leq i\leq e$ and $[a,c_a]$ for $a\in E$. It follows that $I'\not =J'$ because $\sdepth_SI/J=d+1$.
 Then $\sdepth_SI'/J'\leq d+1$  using \cite[Lemma 2.2]{R} and so
$\depth_SI'/J'\leq d+1$ by Conjecture \ref{c} applied for $r-e<r$. But the last term of the above sequence has depth $>d$ because
$x_t$ does not annihilate $f_i$ for $i\in [e]$. With the Depth Lemma we get $\depth_SI/J\leq d+1$.
\hfill\ \end{proof}

Next we give a variant of the above lemma.

\begin{Lemma} \label{el2} Suppose that $r>2$, $E=\emptyset$, $C\subset (W)$ and
there exists  $t\in [n]$, $t\not\in  \cup_{i\in [r]} \supp f_i$  such that $x_tw_{ij}\in C$ for some $1\leq i<j\leq r$.
If Conjecture \ref{c} holds for  $r'\leq r-2$  and
 $\sdepth_SI/J=d+1$,
then $\depth_SI/J\leq d+1$.
\end{Lemma}
\begin{proof} We follow the proof of the above lemma, skipping the first part since we have already $C\subset (W)$. Note that in our case $x_tf_i,x_tf_j\in B$ and so $e\geq 2$. Thus $I'$ is generated by at most $(r-2)$ monomials of degrees $d$ and some others of degrees $\geq d+1$. Therefore, Conjecture \ref{c} holds for $I'/J'$ and so the above proof works in our case.
\hfill\ \end{proof}

For $r\leq 3$ the following lemma is part from the proof of \cite[Lemma 3.2]{AP} but not in an explicit way. Here we try to formalize better the arguments in order to apply them when $r=4$.
\begin{Lemma} \label{dep} Suppose that  $r\leq 4$  and for each $i\in [r] $ there exists $c_i\in C\cap (f_i)$ such that the intervals $[f_i,c_i]$, $i\in [r]$ are disjoint. Then $\depth_SI/J\geq d+1$.
\end{Lemma}
\begin{proof} The proof consists of an induction part dealing with the case $C\not \subset (W)$ followed by a case analysis covering the case $C \subset (W)$.

{\bf Case 1}, $C\not \subset (W)$

Suppose that there exists $c\in C\setminus (W)$, let us say $c\in (f_1)\setminus (f_2,\ldots,f_r)$. Then $[f_1,c]$ is disjoint with respect to $[f_i,c_i]$, $1<i\leq r$ and we may change $c_1$ by $c$, that is we may suppose that $c_1\in (f_1)\setminus (f_2,\ldots,f_r)$. Let $B\cap [f_1,c_1]=\{b,b'\}$ and $L=(f_2,\ldots,f_r,B\setminus
\{b,b',E\})$. In the following exact sequence
$$0\to L/(J\cap L)\to I/J\to I/(J,L)\to 0$$
the first term has depth $\geq d+1$ by induction hypothesis and the last term is isomorphic with $(f_1)/((J,L)\cap (f_1))$ and has depth $\geq d+1$ because $b\not \in (J,L)$. Thus $\depth_SI/J\geq d+1$ by the Depth Lemma.

{\bf Case 2}, $r=2$

In this case, note that one from $c_1,c_2$ is not in $(W)=(w_{12})$, that is we are in the above case. Indeed, if $c_1\in (W)$ then either $c_1=w_{12}$ and so $c_2$ cannot be in $(W)$, or $c_1=x_jw_{12}$ and then $w_{12}\in [f_1,c_1]$ cannot divide $c_2$ since the intervals are disjoint.

From now on assume that $r>2$.

{\bf Case 3}, $c_1\in (w_{12})$,  $f_i\not |c_1$ for $i>2$ and $c_i\not \in (w_{12})$ for  $1<i\leq r$.

First suppose that $w_{12}\in B$.  We have  $c_1=x_jw_{12}$ for some $j$ and we see that  $b=f_1x_j\not\in (f_2,\ldots,f_r)$.  Set
$T=(f_2,\ldots,f_r,B\setminus
\{b,E\})$. In the following exact sequences
$$0\to T/(J\cap T)\to I/J\to I/(J,T)\to 0$$
$$0\to (w_{12})/(J\cap (w_{12}))\to T/(J\cap T)\to T/((J,w_{12})\cap T)\to 0$$
the last terms have depth $\geq d+1$ since $b\not\in (J,T)$ and using the induction hypothesis in the second situation. As the first term of the second sequence has depth $\geq d+1$ we get $\depth_ST/(J\cap T)\geq d+1$ and so $\depth_SI/J\geq d+1$ using the Depth Lemma in both exact sequences.

If $w_{12}\in C$ then both monomials $b,b'$ from $B\cap [f_1,c_1]$ are not in $(f_2,\ldots,f_r)$ and the above proof goes with $b'$ instead $w_{12}$.

{\bf Case 4}, $r=3$.

By Case 1 we may suppose that $C\subset (W)$. Then $w_{12},w_{13},w_{23}$ are different because otherwise only one $c_i$ can be in $(W)$. We may suppose that $c_1\in (w_{12}) $, $c_2\in (w_{23}) $, $c_3\in (w_{13}) $, because each $c_i $ is a multiple of one $w_{ij}$ which can be present just in one interval since these are disjoint. If $f_3|c_1$ then $w_{13}$ is present in both intervals $[f_1,c_1]$, $[f_3,c_3]$. If let us say $w_{12}\in C$, then $c_2,c_3\not\in (w_{12})$ because $c_3\not=c_1\not=c_2$. Thus we are in Case 3.

  If $w_{12}\in B$ and $c_2,c_3\not\in (w_{12})$  then we are in Case 3. Otherwise, we may suppose that either $c_2\in (w_{12})$, or $c_3\in (w_{12})$. In the first case, we have $w_{12}$ in both intervals $[f_1,c_1]$, $[f_2,c_2]$, which is false. In the second case, we have also $w_{23}$ present in both intervals $[f_2,c_2]$, $[f_3,c_3]$, again false.

{\bf Case 5}, $r=4$, $c_1\in (w_{12})$, $w_{12}\in B$, $f_i\not |c_1$ for $2<i\leq 4$, $c_3\in  (w_{12})$.

It follows that $c_3\in (w_{23})$.  Thus $c_2\not\in (w_{23})$, that is $f_3\not |c_2$, because otherwise   the intervals $[f_2,c_2]$, $[f_3,c_3]$ will contain $w_{23}$, which is false. If $c_2\in (w_{12})$ then the intervals $[f_1,c_1]$, $[f_2,c_2]$ will contain $w_{12}$.
 It follows that $c_2\in (w_{24})$. Note that $c_4\not\in (w_{24})$ because otherwise $w_{24}$ belongs to $[f_2,c_2]\cap [f_4,c_4]$. If $c_3\not\in (w_{24})$ then we are in Case 3 with $w_{24}$ instead $w_{12}$ and $c_2$ instead $c_1$.

Remains to see the case when $c_3\in (f_1)\cap (f_2)\cap (f_3)\cap (f_4)$. Then $c_4\not\in (f_3)$ because otherwise $w_{34}$ is in $[f_3,c_3]\cap [f_4,c_4]$. In the exact sequence
$$0\to (f_3)/(J\cap (f_3))\to I/J\to I/(J,f_3)\to 0$$
the last term has depth $\geq d+1$ by induction hypothesis. The first term has depth $\geq d+1$ since for example $w_{23}\not \in J$. By the Depth Lemma we get $\depth_SI/J\geq d+1$.

{\bf Case 6}, $r=4$, the general case.

Since $|W|\leq 6$ there exist an interval, let us say $[f_1,c_1]$, containing just one $w_{ij}$, let us say $w_{12}$. Thus no $f_i$, $2<i\leq 4$ divides $c_1$.
If $w_{12}\in C$ then no $c_i$, $i>1$ belongs to $(w_{12})$ because otherwise $c_i=c_1$.
If $w_{12}\in B$ and one $c_i\in (w_{12})$, $i>1$ then we must have $i=2$ because otherwise we are in Case 5. But
if $c_2\in (w_{12})$ then $w_{12}$ is present in both intervals $[f_1,c_1]$, $[f_2,c_2]$, which is false. Thus $c_i\not\in (w_{12})$ for all $1<i\leq 4$, that is  Case 3.
\hfill\ \end{proof}

\begin{Remark}{\em When $r>4$ the statement of the above lemma is not valid anymore, as shows the following example.}
\end{Remark}

\begin{Example} \label{e3} {\em Let $n=5$, $d=1$, $I=(x_1,\ldots,x_5)$, $$J=(x_1x_3x_4,x_1x_2x_4,x_1x_3x_5,x_2x_3x_5,x_2x_4x_5).$$ Set $c_1=x_1x_2x_3$, $c_2=x_2x_3x_4$, $c_3=x_3x_4x_5$, $c_4=x_1x_4x_5$, $c_5=x_1x_2x_5$. We have $C=\{c_1,\ldots,c_5\}$ and $B=W$. Thus $s=2r$ and $\sdepth_SI/J=3$ because we have a partition on $I/J$  given by the intervals $[x_i,c_i]$, $i\in [5]$. But $\depth_SI/J=1$ because of the following exact sequence
$$0\to I/J\to S/J\to S/I\to 0$$
where the last term has depth $0$ and the middle $\geq 2$.}
\end{Example}

The proposition below is an extension of \cite[Lemma 3.2]{AP}, its proof is given in the next section.

\begin{Proposition}\label{ml} Suppose that the following conditions hold:
\begin{enumerate}
\item{} $r=4$, $8\leq s\leq q+4$,
\item{} $C\subset (\cup_{i,j\in [4],i\not=j} (f_i)\cap (f_j))\cup ((E)\cap (f_1,\ldots,f_4))\cup (\cup_{a,a'\in E, a\not =a'}(a)\cap (a'))$,
\item{}  there exists $b\in (B\cap (f_1))\setminus (f_2,f_3,f_4)$ such that $\sdepth_SI_{b}/J_b\geq d+2$ for $I_b=(f_2,\ldots,f_r,B\setminus\{b\})$, $J_b=J\cap I_{b}$,
\item{} the least common multiple $\omega_1$ of $f_2,f_3,f_4$ is not in $(C_3\setminus W)\cap (E)$ (see Example \ref{exe}).
\end{enumerate}
Then either  $\sdepth_SI/J\geq d+2$, or there exists a nonzero   ideal $I'\subsetneq I$ generated  by a subset of $\{f_1,\ldots,f_4\}\cup B$  such that  $\depth_SI/(J,I')\geq d+1$ and either $\sdepth_S I'/J'\leq d+1$ for  $J'=J\cap I'$ or  $\depth_S I'/J'\leq d+1$ .
\end{Proposition}

\begin{Proposition} \label{m'} Conjecture \ref{c} holds for $r=4$ when  the least common multiples  $\omega_i$ of $f_1,\ldots,f_{i-1},f_{i+1},\ldots,f_4$,  $i\in [4]$ are not in $(C_3\setminus W)\cap (E)$. In particular, Conjecture \ref{c} holds when $r=4$ and $E=\emptyset$.
\end{Proposition}
\begin{proof}

By Theorems \cite[Theorem 1.3]{P1}, \cite[Theorem 2.4]{Sh} (more precisely the particular forms given in \cite[Theorems 0.3, 0.4]{AP}) we may suppose that $8=2r\leq s\leq q+4$ and we    may assume  that    $E$ contains only monomials of degrees $d+1$ by \cite[Lemma 1.6]{PZ}.
 We  may assume that there exists  $b\in B\cap (f_1,\ldots,f_4)$ which is not in $W$  because otherwise $ B\cap (f_1,\ldots,f_4)\subset B\cap W$ and therefore
 $|B\cap (f_1,\ldots,f_4)|\leq |B\cap W|\leq 6$.
   By \cite[Theorem 2.4]{Sh} this implies the depth $\leq d+1$ of the first term  of   the exact sequence
 $$0\to (f_1,\ldots,f_r)/(J\cap (f_1,\ldots,f_r))\to I/J\to (E)/((J,f_1,\ldots,f_r)\cap (E))\to 0$$
and then the middle has depth $\leq d+1$ too using
 the Depth Lemma.

   Renumbering $f_i$ we may suppose that there exists $b\in (f_1)\setminus (f_2,\ldots,f_4)$.
  As in the proof of \cite[Theorem 1.7]{AP} we may suppose that the first term of the exact sequence
 $$0\to I_b/J_b\to I/J\to I/(J,I_b)\to 0$$
 has sdepth $\geq d+2$. Otherwise it has depth $\leq d+1$ by Theorem \ref{ap}. Note that the last term is isomorphic with $(f_1)/((f_1)\cap (J,I_b))$ and it has depth $\geq d+1$ because $b\not\in (J,I_b)$.  Then  the middle term of the above exact sequence has depth  $\leq d+1$ by the Depth Lemma.

  Thus we may assume that the condition (3) of Proposition \ref{ml} holds. Also we may apply  \cite[Lemma 1.1]{PZ2} and see that the condition (2) of Proposition \ref{ml} holds. Applying Proposition \ref{ml}  we get either $\sdepth_SI/J\geq d+2$ contradicting our assumption, or
 there exists a nonzero   ideal $I'\subsetneq I$ generated by a subset $G$ of $ B$, or by $G$ and a subset of $\{f_1,\ldots,f_4\}$  such that  $\sdepth_S I'/J'\leq d+1$ for  $J'=J\cap I'$ and $\depth_SI/(J,I')\geq d+1$.
  In the last case we see that  $\depth_S I'/J'\leq d+1$  by Theorem \ref{ap}, or by induction on $s$, and so $\depth_SI/J\leq d+1$ applying  in the following exact sequence
$$0\to I'/J'\to I/J\to I/(J,I')\to 0$$
 the Depth Lemma.
\hfill\ \end{proof}

  \section{Proof of Proposition \ref{ml}}

Since   $\sdepth_SI_{b}/J_{b}\geq d+2$ by (3),  there exists a partition  $P_b$ on $I_{b}/J_{b}$ with sdepth $d+2$. We may choose $P_b$ such that each interval starting with a squarefree monomial of degree $d$, $d+1$ ends with a  monomial of $C$. In $P_{b}$ we  have three disjoint intervals $[f_2,c'_2]$, $[f_3,c'_3]$, $[f_4,c'_4]$. Suppose that $B\cap [f_i,c'_i]=\{u_i,u'_i\}$, $1<i\leq 4$.
For all $b' \in B\setminus\{b,u_2,u'_2,\ldots,u_4,u'_4\}$ we have an interval $[b',c_{b'}]$. We define $h:
B\setminus\{b,u_2,u'_2,\ldots,u_4,u'_4\}\to C$ by  $b'\longmapsto c_{b'}$. Then $h$ is an injection and $|\Im h|= s-7\leq q-3$.

   We follow the proofs of \cite[Lemmas 3.1, 3.2]{AP}.
 A sequence $a_1,\ldots, a_k$ is called a {\em path} from $a_1$ to $a_k$ if the following statements hold:

   (i) $a_l\in B\setminus \{b,u_2,u_2',\ldots,u_4,u_4'\}$, $l\in [k]$,

   (ii) $a_l\not=a_j$ for $1\leq l<j\leq k$,

   (iii) $a_{l+1}|h(a_l)$  for all $1\leq l<k$.

    This path is {\em weak}  if $h(a_j)\in (b,u_2,u_2',\ldots,u_4,u'_4)$ for some $j\in [k]$. It is {\em bad} if $h(a_j)\in (b)$ for some $j\in [k]$ and it is {\em maximal} if   all divisors from $B$ of $h(a_k)$ are in $\{b,u_2,u_2',\ldots,u_4,u'_4,a_1,\ldots,a_k\}$. We say that the above path {\em starts with} $a_1$.
Note that here the notion of path is more general than the notion of path used in \cite{PZ2} and \cite{AP}.

By hypothesis $s\geq 8$ and there exists
$a_1\in B\setminus \{b,u_2,u_2',\ldots,u_4,u'_4\}$.  We construct below, as an example,  a path with $k>1$. By recurrence choose if possible $a_{p+1}$  to be a divisor from $B\setminus \{b,u_2,u'_2,\ldots,u_4,u_4',a_1,\ldots,a_{p}\}$ of $m_p=h(a_p)$, $ p \geq 1$. This construction ends at step $p=e$ if all divisors from $B$ of $m_{e}$ are in $\{b,u_2,u'_2,\ldots,u_4,u_4',a_1,\ldots,a_{e}\}$. This is a maximal path. If one $m_p\in  (u_2,u'_2,\ldots,u_4,u_4')$ then the constructed path is weak. If one $m_{p}\in (b)$ then this path is  bad.

We start the proof with some helpful lemmas.
\begin{Lemma} \label{i0} $P_b$ could be changed in order to have the following properties:
\begin{enumerate}
\item{} For all $1<i<j\leq 4$ with $u_i,u_j\not\in W$ and
 $w_{ij}\in B\setminus \{u_2,u'_2,\ldots,u_4,u'_4\}$ it holds that $h(w_{ij}) \not\in (u_i)\cap (u_j)$,
 \item{} For each $1\leq i<j\leq 4$ with $u_j\in W$,  $u'_j\not\in W$ and
 $w_{ij}\in B\setminus \{u_2,u'_2,\ldots,u_4,u'_4\}$ it holds that $h(w_{ij})\not\in (u_j)$ and if  $h(w_{ij})\in (u'_j)$ then $i>1$,
 \item{} For each $1\leq i<j\leq 4$ with $u_j,u'_j\not\in W$ and
 $w_{ij}\in B\setminus \{u_2,u'_2,\ldots,u_4,u'_4\}$ it holds that $h(w_{ij})\not\in (u_j,u'_j)$.
 \end{enumerate}
 \end{Lemma}
 \begin{proof}

Suppose that $w_{ij}\in B\setminus\{u_2,u'_2,\ldots,u_4,u'_4\}$ and $h(w_{ij})\in (u_i)$ for some $2\leq i\leq 4$ and $j\in [4]$, $j\not=i$. We have  $h(w_{ij})=x_lw_{ij}$ for some $l\not\in \supp w_{ij}$ and   it follows that $u_i=x_lf_i$. Changing in $P_b$ the intervals $[f_i,c'_i]$, $[w_{ij},h(w_{ij})]$ with $[f_i,h(w_{ij})]$, $[u'_i,c'_i]$  we may assume that the new $u'_i=w_{ij}$. We will apply this procedure several times
 eventually obtaining a partition $P_b$ with
 the above properties. In case (1)  we change in this way $u'_i$ by $ w_{ij}$. Note that the number of elements   among $\{u_2,u'_2,\ldots,u_4,u'_4\}$ which are from   $B\cap W$ is either preserved or increases by one. Applying  this procedure several time we get  (1)  fulfilled.

In case (3)
 the above procedure  preserves among $\{u_2,u'_2,\ldots,u_4,u'_4\}$ the former elements which were from   $B\cap W$ and includes a new one $w_{ij}$. After several steps we get fulfilled (3).

For case (2) if $u_j\in W$, $u'_j\not\in W$ and $h(w_{ij})\in (u_j) $ we change as above $u'_j$ by $w_{ij}$. Note that the number of elements   among $\{u_2,u'_2,\ldots,u_4,u'_4\}$ which are from   $B\cap W$ increases by one. If $h(w_{ij})\in (u'_j) $ then we may change in this way $u_j$ by $w_{ij}$. We do this only if $i=1$. Note that the number of elements   among $\{u_2,u'_2,\ldots,u_4,u'_4\}$ which are from   $B\cap W$ is preserved. Our procedure does not affect those $c'_i$ with $u_i,u'_i\in W$
and does not affect the property (1). After several such procedures we get also (2) fulfilled.
\hfill\ \end{proof}
From now on we  suppose that $P_b$ has the properties mentioned in  the above lemma.  Moreover, we fix $a_1\in B\setminus \{b,u_2,u'_2,\ldots,u_4,u'_4\}$ and let $a_1,\ldots,a_p$ be a path which is not bad.  For an $a'\in  B\setminus \{b,u_2,u'_2,\ldots,u_4,u'_4\}$  set  $$T_{a'}=\{b'\in B: \mbox{there\ \ exists\ \ a\ \ path}\ \ a'_1=a',\ldots,a'_{e}\ \  \mbox{ not\ \ bad\ \ with}\ \ a'_{e}=b'\},$$
$U_{a'}=h(T_{a'})$, $G_{a'}=B\setminus T_{a'}$. If $a'=a_1$ we write simply $T_1$ instead $T_{a_1}$ and similarly $U_1$, $G_1$.

\begin{Remark}\label{ea} {\em Any divisor from $B$ of a monomial of $U_1$ is in $T_1\cup \{u_2,u'_2,\ldots,u_4,u'_4\}$.}
\end{Remark}

\begin{Lemma} \label{i} If    no weak path and no  bad path starts with $a_1$  then the conclusion of Proposition \ref{ml} holds.
\end{Lemma}
\begin{proof}
 Assume that $[r]\setminus \{j\in [r]: U_1\cap (f_j)\not=\emptyset\}=\{k_1,\ldots,k_{\nu}\}$ for some  $1\leq k_1<\ldots <k_{\nu}\leq 4$, $0\leq \nu\leq 4$. Set $k=(k_1,\ldots,k_{\nu})$, $I'_k = (f_{k_1},\ldots,f_{k_{\nu}},G_1)$,  $J'_k = I'_k \cap J$, and  $I'_0 = (G_1)$,  $J'_0 = I'_0\cap J$ for $\nu=0$. Note that  all divisors from $B$ of a monomial $c\in U_1$ belong to $T_1$, and  $I'_0\not =0$ because $b\in I'_0$.  Consider the following exact sequence
$$0\to I'_{k}/J'_{k}\to I/J\to I/(J,I'_{k})\to 0.$$
If $U_1\cap (f_1,\ldots,f_4)=\emptyset$ then the last  term of the above exact sequence given for $k=(1,\ldots,4)$ has depth $\geq d+1$ and sdepth $\geq d+2$ because  $P_b$ can be restricted to $(T_1)\setminus (J,I'_k)$ since $h(b) \notin I'_{k}$ , for all $ b \in T_1$ (see Remark \ref{ea}). If the first term has sdepth $\geq d+2$ then by  \cite[Lemma 2.2]{R}  the middle term  has  sdepth  $\geq d+2$. Otherwise, take $I'=I'_k$.

If $U_1\cap (f_1,f_2,f_3)=\emptyset$, but there exists $b_4\in T_1\cap (f_4)$,  then set $k=(1,2,3)$. In the following exact sequence
$$0\to I'_k/J'_k\to I/J\to I/(J,I'_k)\to 0$$
the last term has sdepth $\geq d+2$   since $h(b') \notin I'_{k}$ for all $ b' \in T_1$ and we may substitute the interval $[b_4,h(b_4)]$  from the restriction of $P_b$ by $[f_4,h(b_4)]$, the second monomial from  $[f_4,h(b_4)]\cap B$ being also in $T_1$. As above we get either $\sdepth_S I/J \geq d+2$, or  $\sdepth_S I'_k/J'_k \leq d+1$,  $\depth_S I/(J,I'_k) \geq d+1$.

Suppose that
 $U_1\cap (f_{j})\not =\emptyset$ if and only if $\nu<j\leq 4$, for some $0\leq \nu\leq 4$ and set  $k=(1,\ldots,\nu)$.
We omit the  subcases $0<\nu<3$, since they go as in \cite[Lemma 3.2]{AP}, and   consider  only the worst subcase $\nu=0$. Let  $b_{j}\in T_1\cap (f_j)$, $j\in [4]$ and set $c_j=h(b_j)$. For $1\leq l<j\leq 4$ we claim that  we may   choose $b_l\not=b_j$ and such  that one from $c_{l},c_{j}$ is not in $(w_{lj})$. Indeed, if $w_{lj} \not\in B$ and $c_l,c_j\in (w_{lj})$  then necessarily  $c_l=c_j$ and it follows $b_l=b_j=w_{lj}$, which is false.  Suppose that $w_{lj}\in B$ and $c_j=x_pw_{lj}$. Then choose $b_l=x_pf_l\in T_1$. If $c_l=h(b_l)\in (w_{lj})$ then we get $c_l=c_j$ and so $b_l=b_j=w_{lj}$ which is impossible.

We show that we may choose  $b_j\in T_1\cap (f_j)$, $j\in [4]$ such that the intervals $[f_j,c_j]$, $j\in [4]$ are disjoint. Let $C_2$, $C_3$ be as in the beginning of the previous section.
 Set $C'_2=U_1\cap C_2$, $C'_3=U_1\cap C_3$, $C'_{23}=C'_2\cup C'_3$.
 Let ${\tilde c}\in C'_2$, let us say $\tilde c$ is the least common multiple of $f_1,f_2$. Then $\tilde c$ has as divisors two multiples $g_1,g_2$ of $f_1$ and two multiples of $f_2$. If ${\hat c}\in C'_2$ is also a multiple of $g_1$, let us say $\hat c$ is the least common multiple of $f_1,f_3$ then $g_2$ does not divide $\hat c$ and the least common multiple of $f_2,f_3$ is not in $C$. Thus the divisors from $B\setminus E$
 of $\tilde c$, $\hat c$ are at least $7$. Since the divisors from $B\setminus E$ of $\tilde c$, $\hat c$  are in $T_1\setminus E$ we see in this way that $|T_1\setminus E|\geq |C'_2|+3$. If $|C'_2|\not = 0$ then $|C'_3|\leq 1$ and so $|T_1\setminus E|\geq |C'_{23}|+2$. Assume that  $|C'_2|=0$. Then
 $|C'_3|\leq 4$.  Let ${\tilde c}\in C'_3$ be the least common multiple of $f_1,f_2,f_3$ then $w_{12},w_{23},w_{13}$ are the only divisors from $T_1\setminus E$ of $\tilde c$ (this could be not true when $|C'_2|\not =0$ as shows Example \ref{exe}). If ${\hat c}\in C'_3$ is the least common multiple of $f_1,f_2,f_4$ we have also $w_{14},w_{24}$ in $T_1\setminus E$. Similarly, if $|C'_3|\geq  3$ we get also $w_{34}\in T_1\setminus E$. Thus $|T_1\setminus E|\geq |C'_3|+2=|C'_{23}|+2$ also when $|C'_2|=0$.

  Then there exist two different  $b_j\in T_1\cap (f_j)$ such that $c_j=h(b_j)\not \in C'_{23}$ for let us say $j=1,2$ and so each of the intervals $[f_j,c_j]$, $j=1,2$ has at most one monomial from $T_1\cap W$. Suppose the worst subcase when  $[f_1,c_1]$ contains $ w_{12}\in B$, and $[f_2,c_2]$ contains  $w_{2j}\in B$ for some $j\not=2$. First assume that $j\geq 3$, let us say $j=3$.  Then choose as above $b_3\in T_1\cap (f_3)$, $b_4\in T_1\cap (f_4)$ such that $c_3\not\in (w_{23})$, $c_4\not\in (w_{34})$. Then $[f_3,c_3]$ has from $T_1\cap W$ at most $w_{13},w_{34}$ and $[f_4,c_4]$ has from $T_1\cap W$ at most $w_{14},w_{24}$. Thus
 the corresponding intervals are disjoint.

 Otherwise, $j=1$ and we have $c_j=x_{p_j}w_{12}$, $j\in [2]$,  for some $p_j\not \in \supp w_{12}$, $p_1\not =p_2$. Take $b'_1=x_{p_2} f_1$,  $b'_2=x_{p_1} f_2$ and $v_1=h(b'_1)$, $v_2=h(b'_2)$. Then $v_1,v_2$ are  not in $C'_3$ because otherwise $b'_1$, respectively $b'_2$ is in $W$, which is false. Note that
$v_2\not\in (w_{12})$, because otherwise $v_2=x_{p_1}w_{12}=c_1$ which is false since $b_1\not= b'_2$. Similarly $v_1\not \in (w_{12})$. If let us say $v_2\not \in C'_2$ then  we may take $b_2=b'_2$ and we see that for the new $c_2$ (namely $v_2$) the interval $[f_2,c_2]$ contains at most a monomial from $W$, which we assume to be  $w_{23}$ and we proceed as above.  If $v_1,v_2\in C'_2$, we may assume  that $v_1=w_{13}\in C$ and either $v_2=w_{23}\in C$, or $v_2=w_{24}\in C$. In the first case we choose $b_3,b_4$ such that $c_3\not \in (w_{34})$, $c_4\not\in (w_{24})$ and we see that $[f_3,c_3]$ has no monomial  from $W$. Indeed, if $c_3\in (w_{23})$ (the case $c_3\in (w_{13})$ is similar) then $c_3=v_2$, which is false since then
$h(b'_2)=v_2=c_3=h(b_3)$ and so
$b'_2 =b_3\in (w_{23})$, $h$ being injective. Also $[f_4,c_4]$ has at most $w_{14},w_{34}$. Thus taking $b_i=b'_i$, $c_i=v_i$ for $i\in [2]$ we have again the intervals $[f_j,c_j]$, $j\in [4]$  disjoint. Similarly in the second case choose    $b_3,b_4$ such that $c_3\not \in (w_{23})$, $c_4\not\in (w_{34})$ and we see that $[f_3,c_3]$ have at most  $w_{34}$ and $[f_4,c_4]$ have at most $w_{14}$, which is enough, because as above $c_3\not=w_{13}$ and $c_4\not=w_{24}$.

 Next we  replace the intervals $[b_{j},c_{j}]$, $1\leq j\leq 4$  from the restriction of $P_b$ to $(T_1)\setminus (J,I'_0)$ with $[f_j,c_{j}]$, the second monomial from  $[f_j,c_{j}]\cap B$ being also in $T_1$.
 Note  that $I/(J,I'_0)$  has depth $\geq d+1$ by Lemma \ref{dep}.  Thus, as
above we get either $\sdepth_S I/J \geq d+2$, or  $\sdepth_S I'_0/J'_0 \leq d+1$,  $\depth_S I/(J,I'_0) \geq d+1$.
\hfill\ \end{proof}

\begin{Lemma} \label{ii} Let  $a_1,\ldots,a_{e_1}$ be a bad path, $m_j=h(a_j)$, $j\in [e_1]$ and
$m_{e_1}=bx_i$. Suppose that $m_{e_1}\not\in (u_2,u'_2,\ldots,u_4,u'_4)$. Then one of the following statements holds:
\begin{enumerate}
\item{} $\sdepth_SI/J \geq d+2$,
\item{} there exists $a_{e_1+1}\in (B\cap (f_1))\setminus \{b,u_2,u'_2,\ldots,u_4,u'_4\}$ dividing $m_{e_1}$ such that every path $a_{e_1+1},\ldots,a_{e_2}$ satisfies
$\{ a_1,\ldots,a_{e_1}\}\cap \{a_{e_{1}+1},\ldots,a_{e_2}\}=\emptyset.$
\end{enumerate}
 \end{Lemma}
 \begin{proof}
If $a_{e_1}= f_1x_i$ then changing in $P_b$ the interval $[a_{e_1},m_{e_1}]$ by $[f_1,m_{e_1}]$ we get a partition on $I/J$ with sdepth $d+2$. If  $f_1x_i\in \{a_1,\ldots,a_{e_1-1}\}$, let us say   $f_1x_i=a_v$, $1\leq v<e_1$ then we may replace in $P_b$ the intervals $[a_k,m_k], v \leq k \leq e_1$ with the intervals $[a_v,m_{e_1}],[a_{k+1},m_k], v \leq k \leq e_1-1$. Now we see that we have in $P_b$ the interval $[a_v,m_v]$ (the new $m_v$ is the old $m_{e_1}$) and switching it with the interval $[f_1,m_v]$ we get a partition with sdepth $\geq d+2$ for $I/J$.
Thus we may assume that $f_1x_i \notin \{ a_1,...,a_{e_1} \}$. Note that $e_1$ could be also $1$ as in Example \ref{b1} when we take $a_1=x_5x_6$, in this case we take $f_1x_i=x_1x_5$ and $\{x_1x_5,x_2x_5\}$ is a maximal path which is weak but not bad.

  By hypothesis $m_{e_1}\not\in (u_2,u'_2,\ldots,u_4,u'_4)$ and so $f_1x_i\not \in \{u_2,u'_2,\ldots,u_4,u'_4\}$.  Then set $a_{e_1+1}=f_1x_i$
    and let  $a_{e_1+1},\ldots,a_{e_2}$ be a path starting with $a_{e_1+1}$ and set
    $ m_p = h(a_p), p > e_1$.  If $a_p=a_v$ for $v\leq e_1$, $p>e_1$ then  change in $P_b$ the intervals $[a_k,m_k], v \leq k \leq p-1$ with the intervals $[a_v,m_{p-1}],[a_{k+1},m_k], v \leq k \leq p-2$. We have  in the new $P_b$ an interval $[f_1x_i,m_{e_1}]$ and switching it to $[f_1,m_{e_1}]$ we get a partition  with sdepth $\geq d+2$ for $I/J$. Thus we may suppose that  $a_{p+1}\not \in \{b,u_2,u'_2,\ldots,u_4,u'_4,a_1,\ldots, a_p\}$ and so (2) holds.
\hfill\ \end{proof}

\begin{Example} \label{b1} {\em
Let $n=7$, $r=4$, $d=1$, $f_i=x_i$ for $i\in [4]$, $E=\{x_5x_6,x_5x_7\}$,  $I=(x_1,\ldots,x_4,E)$ and $$J=(x_1x_7,x_2x_7,x_3x_7,x_4x_7, x_1x_2x_4,x_1x_2x_6,x_1x_3x_4,x_1x_3x_6,x_2x_3x_4,x_2x_4x_5,$$
$$x_2x_5x_6,x_3x_5x_6,x_4x_5x_6).$$  Then $B=$
$$\{x_1x_2,x_1x_3,x_1x_4,x_1x_5,x_1x_6,x_2x_3,x_2x_4,x_2x_5,x_2x_6,x_3x_4,x_3x_5,x_3x_6,x_4x_5,x_4x_6\}\cup E$$ and $$C=\{x_1x_2x_3,x_1x_2x_5,x_1x_3x_5,x_1x_4x_5,x_1x_4x_6,x_1x_5x_6,x_2x_3x_5,x_2x_3x_6,x_2x_4x_6,$$
$$x_3x_4x_5,x_3x_4x_6,x_5x_6x_7\}.$$ We have $q=12$ and $s=q+r=16$. Take $b=x_1x_6$ and \\ $I_b=(x_2,x_3,x_4,B\setminus \{b\},E)$, $J_b=I_b\cap J$. There exists a partition $P_b$ with sdepth $3$ on $I_b/J_b$ given by the intervals $[x_2,x_1x_2x_3]$, $[x_3,x_1x_3x_5]$, $[x_4,x_1x_4x_6]$, $[x_1x_5,x_1x_2x_5]$, $[x_2x_4,x_2x_4x_6]$, $[x_2x_5,x_2x_3x_5]$, $[x_2x_6,x_2x_3x_6]$, $[x_3x_4,x_3x_4x_5]$, $[x_3x_6,x_3x_4x_6]$,\\  $[x_4x_5,x_1x_4x_5]$, $[x_5x_6,x_1x_5x_6]$,  $[x_5x_7,x_5x_6x_7]$. We have $c'_2=x_1x_2x_3$, $c'_3=x_1x_3x_5$, $c'_4=x_1x_4x_6$ and $u_2=x_2x_3$, $u'_2=x_1x_2$, $u_3=x_3x_5$, $u'_3=x_1x_3$, $u_4=x_1x_4$, $u'_4=x_4x_6$. Take $a_1=x_2x_4$, $m_1=x_2x_4x_6$. This is a weak path but not bad. It can be extended to a maximal one
 $x_2x_4,x_2x_6,x_3x_6,x_3x_4,x_4x_5,x_1x_5,x_2x_5$ which is not bad. Bad paths are for example  $\{x_5x_6\}$,  $ \{x_5x_7,x_5x_6\}$,  $ \{ x_5x_7,x_5x_6,x_1x_5,x_2x_5\} $, the last one being maximal.
   Replacing in $P_b$ the intervals $[x_4,x_1x_4x_6]$, $[x_2x_4,x_2x_4x_6]$ with $[x_4,x_2x_4x_6]$, $[x_1,x_1x_4x_6]$ we get a partition on $I/J$ with sdepth $3$.
}
\end{Example}

\begin{Lemma} \label{ii1} Let  $a_1,\ldots,a_{e_1}$ be a bad path, $m_j=h(a_j)$, $j\in [e_1]$ and
$m_{e_1}=bx_i$. Suppose that $a_{e_1}\in E$ and  $m_{e_1}\in (u_2,u'_2,\ldots,u_4,u'_4)$. Then one of the following statements holds:
\begin{enumerate}
\item{} there exists $a_{e_1+1}\in B\setminus (\{b,u_2,u'_2,\ldots,u_4,u'_4\}\cup E)$ dividing $m_{e_1}$ such that every path $a_{e_1+1},\ldots,a_{e_2}$ satisfies
$\{ a_1,\ldots,a_{e_1}\}\cap \{a_{e_{1}+1},\ldots,a_{e_2}\}=\emptyset,$
\item{} there exist $j$, $2\leq j\leq 4$ and   a new partition $P_b$ of $I_b/J_b$ for which $T_1$ is preserved such that $a_{e_1}\in (f_j)$ and $m_{e_1}\in (u_j,u'_j)$.
\end{enumerate}
 \end{Lemma}

\begin{proof}
 Assume that  $m_{e_1}=x_ib$ for some $i$ and let us say $m_{e_1}\in (u'_2)$. Then $f_1x_i = u'_2=w_{12}$ and so there exists another divisor $\tilde a$ of $m_{e_1}$ from $B\cap (f_2)$ different of $w_{12}$. If ${\tilde a}\in [f_2,c'_2]$ then we get $m_{e_1}=c'_2$, which is false. If $\tilde a$ is not in $\{b,u_2,u'_2,\ldots,u_4,u'_4\}$ then set $a_{e_1+1}={\tilde a}$. If let us say  $\tilde a=u_3$ then ${\tilde a}=w_{23}$ and so $m_{e_1} $ is the least common multiple of $f_1,f_2,f_3$. Clearly, $m_{e_1}\not \in C_3$ because otherwise $b\in W$, which is false. Then $m_{e_1}=w_{13}\in C$ and we may find, let us say another divisor $\hat a$ of $m_{e_1}$ from $B\cap (f_3)$ which is not  $u'_3$ because $m_{e_1}\not =c'_3$. If $\hat a$ is in $\{u_4,u'_4\}$ then we may find an $ a'$ in $B\cap (f_4)$ which is not in $\{u_4,u'_4\}$ because $m_{e_1}\not =c'_4$. Thus in general we may find an  $ a''$ in $B\cap (f_j)$ for some $2\leq j\leq 4$ which is not in $\{b,u_2,u'_2,\ldots,u_4,u'_4\}$ and $m_{e_1}\in (u_j,u'_j)$. Set $a_{e_1+1}= a''$.  Let $ a_{e_1+1},\ldots, a_{e_2}$ be a  path. If we are not in the case (1) then  $a_p=a_v$ for $v\leq e_1$, $p>e_1$ and  change in $P_b$ the intervals $[a_k,m_k], v \leq k \leq p-1$ with the intervals $[a_v,m_{p-1}],[a_{k+1},m_k], v \leq k \leq p-2$.
    Note that the new $a_{e_1}$ is the old $a_{e_1+1}\in (f_j)$, that is  the case (2).
\hfill\ \end{proof}

\begin{Lemma} \label{ii2} Suppose that  $\sdepth_SI/J \leq d+1$. Then there exists  a  partition $P_b$ of $I_b/J_b$ such that  for any $a_1\in  B\setminus \{b,u_2,u'_2,\ldots,u_4,u'_4\}$ and any  bad path $a_1,\ldots,a_{e_1}$ , $m_j=h(a_j)$, $j\in [e_1]$ with
$m_{e_1}=bx_i$ the following statements holds:
\begin{enumerate}
\item{} $m_{e_1}\not \in (u_2,u'_2,\ldots,u_4,u'_4)$,

\item{} there exists $a_{e_1+1}\in B\setminus (\{b,u_2,u'_2,\ldots,u_4,u'_4\}\cup E)$ dividing $m_{e_1}$ such that every path $a_{e_1+1},\ldots,a_{e_2}$ satisfies
$\{ a_1,\ldots,a_{e_1}\}\cap \{a_{e_{1}+1},\ldots,a_{e_2}\}=\emptyset.$
\end{enumerate}
 \end{Lemma}

\begin{proof}
   If for any $a_1\in  B\setminus \{b,u_2,u'_2,\ldots,u_4,u'_4\}$  there exist no bad path starting with $a_1$ there exists nothing to show. If for any such $a_1$ for each    bad path $a_1,\ldots,a_{e_1}$, $m_j=h(a_j)$, $j\in [e_1]$ with
$m_{e_1}\in (b)$ it holds $m_{e_1}\not \in (u_2,u'_2,\ldots,u_4,u'_4)$ then then to get (2) apply Lemma \ref{ii}.  Now suppose that there exists $a_1$ and a bad path $ a_1,\ldots,a_{e_1}$, $m_j=h(a_j)$, $j\in [e_1]$ with let us say
 $m_{e_1} \in (b)\cap (u_2)$.
 If we are not in case (2) then by Lemma \ref{ii1} we may change $P_b$ such that  $T_1$ is preserved,  $a_{e_1}\in (f_j)$ and $m_{e_1}\in (u_j,u'_j)$ for some $2\leq j\leq 4$. Assume that $j=2$ and so $m_{e_1}\in (w_{12})$, let us say $u'_2=w_{12}$.
 Replacing in $P_b$ the intervals $[f_2,c'_2]$, $[a_{e_1},m_{e_1}]$ with $[f_2,m_{e_1}]$, $[u_2,c'_2]$ the new $c'_2$ is the least common multiple of $b$ and $f_2$. Thus there exists no path $a_1,\ldots,a_{e_1}$ with $h(a_{e_1})\in (b)\cap (u_2,u'_2)$ because $h(a_{e_1})\not= c'_2$.
 Applying this procedure several time we  see that there exists no path $a_1,\ldots,a_{e_1}$ with $h(a_{e_1})\in (b)\cap (u_2,u'_2,\ldots,u_4,u'_4)$.
 Then we may apply Lemma \ref{ii} as above.
\hfill\ \end{proof}

\begin{Example}\label{b'}{\em
Let $n=5$,  $I=(x_1,\ldots,x_4)$,  $J=(x_2x_3x_4,x_2x_3x_5,x_2x_4x_5,x_3x_4x_5)$. So $$C=\{x_1x_2x_3,x_1x_2x_4,x_1x_2x_5,x_1x_3x_4,x_1x_3x_5,x_1x_4x_5\},$$ $$B=\{x_1x_2,x_1x_3,x_1x_4, x_1x_5,x_2x_3,x_2x_4,x_2x_5,x_3x_4,x_3x_5,x_4x_5\}.$$ Then $q=6$, $s=10=q+r$. Set $b=x_1x_5$, $a_1=x_2x_5$, $a_2=x_3x_5$, $a_4=x_4x_5$, $m_1=x_1x_2x_5$, $m_2=x_1x_3x_5$, $m_3=x_1x_4x_5$, $c'_2=x_1x_2x_3$, $c'_3=x_1x_3x_4$, $c'_4=x_1x_2x_4$. We have on $I_b/J_b$ the partition $P_b$ given by the intervals $[x_i,c'_i]$, $2\leq i\leq 4$ and $[a_j,m_j]$, $j\in [3]$. Clearly, $P_b$ has sdepth $3$ and $m_i=bx_i$, $2\leq i\leq 4$. Using the above lemma we change in $P_b$ the intervals $[a_{i-1},m_{i-1}]$, $[x_i,c'_i]$  with $[f_i,m_{i-1}]$, $[x_ix_5,c'_i]$ for $2\leq i\leq 4$.  Now we see that all $m$ from the new $U_1$ are not in $(b)\cap (u_2,u'_2,\ldots,u_4,u'_4)$.

We have  $\sdepth_SI/J\leq 2$. If $\sdepth_SI/J=3$ then there exists an interval $[x_1,c]$ with $c\in \{m_1,m_2,m_3\}$. If $c=m_i$ for some $2\leq i\leq 4$ then for any interval $[x_i,c']$ it holds $[x_1,c]\cap [x_i,c']=\{x_1x_i\}$, which is impossible. Also we have $\depth_SI/J\leq 2$ by Lemma \ref{vi}.
 }
\end{Example}

\begin{Remark} \label{ii'} {\em Suppose that $\sdepth_SI/J\leq d+1$. We change $P_b$ as in Lemma \ref{ii2}. Moreover assume that  there exists a bad path $a_{e_1+1},\ldots,a_{e_2}$. Using the same  lemma  we find  $a_{e_2+1}$  such that for each path
$a_{e_2+1},\ldots,a_{e_3}$
one has $\{ a_{e_1+1},\ldots,a_{e_2}\}\cap \{a_{e_{i_2}+1},\ldots,a_{e_3}\}=\emptyset.$ The same argument gives also
$\{ a_1,\ldots,a_{e_1}\}\cap \{a_{e_{i_2}+1},\ldots,a_{e_3}\}=\emptyset.$ Thus we may find some disjoint sets of elements
$\{ a_{e_j+1},\ldots,a_{e_{j+1}}\}$, $j\geq 0$, where $e_0=0$. It follows that after some steps we arrive in the case when for some $l$ there exist no bad path starting with $a_{l+1}$.}
\end{Remark}

\begin{Lemma} \label{iv} Suppose that $\sdepth_SI/J\leq d+1$ and ${\tilde P}_b$ is a partition of $I_b/J_b$ given by Lemma \ref{ii2}. Assume that  no  bad path starts with $a_1$,  $U_1\cap (u_2)\not =\emptyset$
and
there exists a divisor $\tilde a$ in $(B\cap (f_2))\setminus  \{u_2,u'_2,\ldots,u_4,u'_4\}$ of  a monomial $m\in U_1\cap ( u_2)$.
 Then there exist a partition $P_b$ and a (possible bad) path $a_1,\ldots,a_p$  such that   $T_{a_p}\cap \{a_1,\ldots,a_{p-1}\}=\emptyset$, $u_2$ and  $c'_i$, $i=3, 4$ are not changed in $P_b$,  no  bad path starts with $a_p$ and one of the following statements holds:
\begin{enumerate}
\item{} $U_{a_p}\cap (u_2)=\emptyset$,
\item{} $U_{a_p}\cap (u_2)\not =\emptyset$
and  there exists $b_2\in T_{a_p}\cap (f_2)$ with $h(b_2)\in (u_2)$,
\item{}  $U_{a_{p}}\cap (u_2)\not =\emptyset$ and every monomial of  $U_{a_{p}}\cap (u_2)$ has
 all its divisors from $B\cap (f_2)$  contained in  $\{u_2,u'_2,\ldots,u_4,u'_4\}$.

\end{enumerate}
Moreover, if
also $U_1\cap (u'_2)\not =\emptyset$, then we may choose  $P_b$ and the path $a_1,\ldots,a_p$  such that either $U_{a_p}\cap (u'_2) =\emptyset$ when there exists a bad path starting with a divisor from $B\setminus \{u_2,u'_2,\ldots,u_4,u'_4\}$ of $c'_2$, or otherwise $u'_2\in T_{a_p}$ and $c'_2=h(u'_2)$.
\end{Lemma}

\begin{proof}  Let $a_1,\ldots, a_{e}$ be a weak path,  $m_j=h(a_j)$, $j\in [e]$  such that $m_{e}=m$. If $a_e={\tilde a}$ then take $b_2=a_e$.
If $a_e\not={\tilde a}$ but there exists  $1\leq v<e$ such that
 $a_v={\tilde a}$.
 Then we may replace in $P_b$ the intervals $[a_{p},m_{p}], v \leq p \leq e$ with the intervals $[a_v,m_{e}],[a_{p+1},m_p], v \leq p < e$. The old $m_{e}$ becomes the new $m_v$, that is we reduce to the above case when $v=e$.

  Now assume that there exist no  such $v$   but there exists a path $a_{e+1}={\tilde a},\ldots,a_{l}$ such that $m_{l}=h(a_{l})\in (a_{v'})$ for some $v'\in [e]$. Then
 we replace in $P_b$ the intervals $[a_{j},m_{j}], v' \leq j \leq l$ with the intervals $[a_{v'},m_{l}]$, $[a_{j+1},m_j]$, $v' \leq j < l$. The new $m_{e+1}$ is the old $m_{e}$  but the new $a_{e+1}$ is the old $a_{e+1}$ and we may proceed  as above.

 Finally, suppose that no path starting with $a_{e+1}$
contains an element from \\
$\{a_1,\ldots,a_{e}\}$.  Taking $p=e+1$  we see that $m\not\in U_{a_p}\cap (u_2)$. If there exists another monomial $m'$ like $m$  then we repeat this procedure and after a while we may get (2), or (3).

Remains to see what happens when we have also $U_{a_p}\cap (u'_2) \not=\emptyset$.  Assume that there exist no bad path starting with a divisor of $c'_2$ from $B\setminus \{u_2,u'_2,\ldots,u_4,u'_4\}$. Then changing in $P_b$  the intervals $[b_2,h(b_2]$, $[f_2,c'_2]$ with $[f_2,h(b_2)]$, $[u'_2,c'_2]$
we see that there exists a path $a_1,\ldots,a_k$,  which is not bad, such that the  old $u'_2=a_k$.
We may complete   $T_{a_p}$ such that $a_k\in T_{a_p}$ and all divisors from $B$ of $c'_2$ which are not in $\{u_2,b_2,u_3,u_3',u_4,u'_4\}$ belong to $T_{a_p}$. For this aim we
   complete  $T_{a_p}$  with the elements connected by a path with $u'_2$ (see Example \ref{b2}).

    Next suppose that there exists a bad path  $a_k=u'_2,\ldots,a_l$ with $h(a_l)\in (b)$.
      We may assume that ${\tilde P}_b$ is given by Lemma \ref{ii2} and so there exist no multiple of $b$ in $U_1\cap (u_2,u'_2,u_3,u'_3,u_4,u'_4)$.  Note that $u''_2=b_2$  the new $u'_2$ considered above  has no multiple in $U_1\cap (b)$ because $b_2\in U_1$. By Lemma \ref{ii} there exists $a_{l+1}\in B\setminus \{b,u_2,u''_2,u_3,u'_3,u_4,u'_4\}$ dividing $h(a_l)$ such that every path $a_{l+1},\ldots,a_{l_1}$ satisfies
$\{ a_1,\ldots,a_{l}\}\cap \{a_{l+1},\ldots,a_{l_1}\}=\emptyset.$ Using Remark \ref{ii'} if necessary  we  have  $T_{a_{p'}}\cap \{a_1,\ldots,a_{p'-1}\}=\emptyset$ for some $p'>l$,  and the above    situation will not appear, that is the old $u'_2$ will not divide anymore a monomial from $U_{a_{p'}}\cap (u_2,u''_2,u_3,u'_3,u_4,u'_4)$. It is also possible that $u_2$ will not divide  a monomial from $U_{a_{p'}}$.
\hfill\ \end{proof}

The following bad example  is similar to \cite[Example 3.3]{AP}.

\begin{Example} \label{b2} {\em
Let $n=7$, $r=4$, $d=1$, $f_i=x_i$ for $i\in [4]$, $E=\{x_5x_6,x_5x_7\}$,  $I=(x_1,\ldots,x_4,E)$ and $$J=(x_1x_7,x_2x_4,x_2x_6,x_2x_7,x_3x_6,x_3x_7,x_4x_6,x_4x_7,x_3x_4x_5).$$  Then $B=\{x_1x_2,x_1x_3,x_1x_4,x_1x_5,x_1x_6,x_2x_3,x_2x_5,x_3x_4,x_3x_5,x_4x_5\}\cup E$ and $$C=\{x_1x_2x_3,x_1x_2x_5,x_1x_3x_4,x_1x_3x_5,x_1x_4x_5,x_1x_5x_6,x_2x_3x_5,x_5x_6x_7\}.$$ We have $q=8$ and $s=q+r=12$. Take $b=x_1x_6$ and \\ $I_b=(x_2,x_3,x_4,B\setminus\{b\},E)$, $J_b=I_b\cap J$. There exists a partition $P_b$ with sdepth $3$ on $I_b/J_b$ given by the intervals $[x_2,x_1x_2x_3]$, $[x_3,x_1x_3x_4]$, $[x_4,x_1x_4x_5]$, $[x_1x_5,x_1x_3x_5]$, $[x_2x_5,x_1x_2x_5]$, $[x_3x_5,x_2x_3x_5]$, $[x_5x_6,x_1x_5x_6]$,  $[x_5x_7,x_5x_6x_7]$. We have $c'_2=x_1x_2x_3$, $c'_3=x_1x_3x_4$, $c'_4=x_1x_4x_5$ and $u_2=x_1x_2$, $u'_2=x_2x_3$, $u_3=x_3x_4$, $u'_3=x_1x_3$, $u_4=x_1x_4$, $u'_4=x_4x_5$. Take $a_1=x_1x_5$, $a_2=x_3x_5$, $a_3=x_2x_5$. This gives a maximal weak path but not bad and defines
 $T_1=\{x_1x_5,x_3x_5,x_2x_5\}$, $U_1=\{x_1x_3x_5,x_2x_3x_5,x_1x_2x_5\}$.

  As in the above lemma we may change in $P_b$ the intervals $[x_2,x_1x_2x_3]$, $[x_2x_5,x_1x_2x_5]$ with $[x_2,x_1x_2x_5]$, $[x_2x_3,x_1x_2x_3]$. Note that the old $u'_2$ is not anymore in $[f_2,c'_2]$ and divides $x_2x_3x_5\in U_1$. Moreover, we have the path $\{a_1,x_1x_5,x_3x_5,x_2x_3\}$ and so we must take $T'_1=(T_1\cup \{x_2x_3\})\setminus \{x_2x_5\}$, $ U'_1=(U_1\cup \{x_1x_2x_3\})\setminus \{x_1x_2x_5\}$ as it is hinted in the above proof. The new $u_2,u'_2$ are all divisors of $x_1x_2x_5$ - the new $c'_2$, which are not in $T'_1$. However, this change of $P_b$ was not necessary because the new $u_2,u'_2,u'_3$ are all divisors from $B$ of the old $c'_2$ (see Remark \ref{M'} and Example \ref{bad}). The same thing is true for $c'_3$ and $c'_4$ has all divisors from $B$ among $\{a_1,u_4,u'_4\}$.}
\end{Example}

\begin{Remark}\label{iv'} {\em Suppose that in Lemma \ref{iv} the partition ${\tilde P}_b$  satisfies also the property (1)  mentioned in Lemma \ref{i0}. If ${\tilde a}=w_{2i}$ for some $i=3,4$ then $m\not \in (u_i,u'_i)$. In particular $b_2\not =w_{23},w_{24}$.}
\end{Remark}

\begin{Lemma}\label{iv''} Assume that $U_{a_{p}}\cap (u_2)\not =\emptyset$ and a monomial $m$ of  $U_{a_{p}}\cap (u_2)$ has
 all its divisors from $B\cap (f_2)$  contained in  $\{u_2,u'_2,\ldots,u_4,u'_4\}$. Then one of the following statements holds:
\begin{enumerate}
\item{} $m$ has a divisor ${\tilde a}_i \in (B\cap (f_i))\setminus \{u_2,u'_2,\ldots,u_4,u'_4\}$ for some $i=3,4$,
\item{} $m\in C_3\setminus W$ and it  is the least common multiple of $f_2,f_3,f_4$.
\end{enumerate}
\end{Lemma}
\begin{proof}  There exists a divisor ${\hat a}\not\in \{u_2,u'_2\}$ of $m$ from $B\cap (f_2)$, otherwise $m=c'_2$. By our assumption we have let us say ${\hat a}=u_3=w_{23}$.
Then there exists a divisor $a'\not =u_3$ from $B\cap (f_3)$. If $a'\not \in \{u_2,u'_2,\ldots,u_4,u'_4\}$ then we are in (1). Otherwise, $a'=u_4=w_{34}$.
 If $m \in W$ then $m=w_{24}\in C_2$ and there exists a divisor of $m$ from
$ (B\cap (f_4))\setminus \{u_2,u'_2,\ldots,u_4,u'_4\}$, that is (1) holds. Thus we may suppose that   $m\not\in W$ and all its divisors from $B\setminus E$ are $w_{23},w_{34},w_{24}$, that is  $m$ is in (2).
\hfill\ \end{proof}

\begin{Remark} {\em Assume that in the above lemma $m$ has the form given in Example \ref{exe}. Then $m\not\in \{c'_2,c'_3,c'_4\}$ and so necessarily $w_{12},w_{13},w_{14}$ are divisors of $m$ from $B\setminus\{u_2,u'_2,\ldots,u_4,u'_4\}$, that is $m$ is in case (1).}
\end{Remark}
\begin{Lemma} \label{v}  Suppose that $\sdepth_SI/J\leq d+1$ and ${\tilde P}_b$ is a partition of $I_b/J_b$ given by Lemma \ref{ii2}. Assume that   ${\tilde P}_b$  satisfies also the properties  mentioned in Lemma \ref{i0}
 and no  bad path starts with $a_1$.
 Then there exist a partition $P_b$ which satisfies the properties mentioned in Lemma \ref{i0} and a (possible bad) path $a_1,\ldots,a_p$  such that $T_{a_p}\cap \{a_1,\ldots,a_{p-1}\}=\emptyset$,  no  bad path starts with $a_p$,  and for every $i=2,3,4$
 such that   there exists a divisor ${\tilde a}_i$ in $(B\cap (f_i))\setminus  \{u_2,u'_2,\ldots,u_4,u'_4\}$ of  a monomial from $ U_1\cap ( u_i)$,
   one of the following statements holds:
\begin{enumerate}
\item{} $U_{a_p}\cap (u_i)=\emptyset$,
\item{} $U_{a_p}\cap (u_i)\not =\emptyset$
and  there exists $b_i\in T_{a_p}\cap (f_i)$ with $h(b_i)\in (u_i)$,
\item{}  $U_{a_{p}}\cap (u_i)\not =\emptyset$ and every monomial of  $U_{a_{p}}\cap (u_i)$ has
 all its divisors from $B\cap (f_i)$  contained in  $\{u_2,u'_2,\ldots,u_4,u'_4\}$.
\end{enumerate}

Moreover, these possible $b_i$ are different and if for some $i=2,3,4$ it holds also
 $U_1\cap (u'_i)\not =\emptyset$, then we may choose  $P_b$ and the path $a_1,\ldots,a_p$  such that either $U_{a_p}\cap (u'_i) =\emptyset$ when there exists a bad path starting with a divisor from $B\setminus \{u_2,u'_2,\ldots,u_4,u'_4\}$ of $c'_i$, or otherwise $u'_i\in T_{a_p}$ and $h(u'_i)$ is the old $c'_i$.
\end{Lemma}

\begin{proof}
Suppose that  there exists a divisor ${\tilde a}_2$ in $(B\cap (f_2))\setminus  \{u_2,u'_2,\ldots,u_4,u'_4\}$ of  a monomial from $ U_1\cap ( u_2)$ with respect of ${\tilde P}_b$.  Using Lemma \ref{iv}  we find  a partition $P_b$ and  a (possible bad) path $a_1,\ldots,a_{p_1}$  such that    $T_{a_{p_1}}\cap \{a_1,\ldots,a_{p_1-1}\}=\emptyset$,  no  bad path starts with $a_{p_1}$ and one of the following statements holds:

$j_2)$ $U_{a_{p_1}}\cap (u_2)=\emptyset$,

 $j'_2)$  $U_{a_{p_1}}\cap (u_2)\not =\emptyset$  and  there exists $b_2\in T_{a_{p_1}}\cap (f_2)$ with $h(b_2)\in (u_2)$,

 $j''_2)$ $U_{a_{p_1}}\cap (u_2)\not =\emptyset$ and every monomial of  $U_{a_{p_1}}\cap (u_2)$ has
 all its divisors from $B\cap (f_2)$  contained in  $\{u_2,u'_2,\ldots,u_4,u'_4\}$.

 Moreover, if
also $U_1\cap (u'_2)\not =\emptyset$, then we may choose  $P_b$ and the path $a_1,\ldots,a_{p_1}$  such that either $U_{a_{p_1}}\cap (u'_2) =\emptyset$ when there exists a bad path starting with a divisor from $B\setminus \{u_2,u'_2,\ldots,u_4,u'_4\}$ of $c'_2$, or otherwise $u'_2\in T_{a_{p_1}}$ and $c'_2=h(u'_2)$.
After a small change  we may suppose that $P_b$  satisfies the properties of Lemma \ref{i0} and so $b_2\not = w_{23},w_{24}$.

If $U_{a_{p_1}}\cap (u_3,u_4)=\emptyset$ then we are done.
Now assume that there exists a  divisor ${\tilde a}_3$  in $B\cap (f_3)\setminus  \{u_2,u'_2,\ldots,u_4,u'_4\}$ of a monomial $m\in U_{a_{p_1}}\cap (u_3)$,
 let us say $m=m_e$ for some path $a_{p_1},\ldots, a_e$. If $a_e={\tilde a}_3$, or $a_e\not ={\tilde a}_3$ but there exists a path $a_{e+1}={\tilde a}_3,\ldots,a_k$ with $a_k=a_v$ for some $v\leq e$ then  we change $P_b$ as in the proof of  Lemma \ref{iv} to replace $c'_3$ by $m$. Clearly, $c'_2,c'_3$ satisfy (2) for $i=2,3$.  Otherwise, if $a_e\not ={\tilde a}_3$ but there exists no path $a_{e+1}={\tilde a}_3,\ldots,a_k$ with $a_k=a_v$ for some $v\leq e$, apply again the quoted lemma with $c'_3$. We get a  (possible bad) path $a_{p_1},\ldots,a_{p_2}$ with $p_2>p_1$  such that   $T_{a_{p_2}}\cap \{a_1,\ldots,a_{p_2-1}\}=\emptyset$,  no  bad path starts with $a_{p_2}$ and one of the following statements holds:

$j_3)$ $U_{a_{p_2}}\cap (u_3)=\emptyset$,

 $j_3')$ $U_{a_{p_2}}\cap (u_3)\not =\emptyset$  and  there exists $b_3\in T_{a_{p_2}}\cap (f_3)$ with $h(b_3)\in (u_3)$,

$j_3'')$ $U_{a_{p_2}}\cap (u_3)\not =\emptyset$ and every monomial $m\in U_{a_{p_2}}\cap (u_3)$ has
 all its divisors from $B\cap (f_3)$  contained in  $\{u_2,u'_2,\ldots,u_4,u'_4\}$.

If we also have $U_1\cap (u'_3)\not =\emptyset$ then it holds a similar statement as in case $i=2$. Note that $b_2\not =b_3$ since $b_2\not =w_{23}$ by Remark \ref{iv'} and so $h(b_2)\not =h(b_3)$.
Very likely meanwhile the corresponding  statements of $j_2)$, $j_2')$, $j_2'')$ do not hold anymore because we could have $b_2\not\in T_{a_{p_2}}$. If there exists another ${\tilde a}_2$  we apply again Lemma \ref{iv} with $c'_2$ obtaining a new partition $P_b$ and a path $a_{p_2},\ldots,a_{p_3}$ for which this situation is repaired. If  now $c'_3$ does not satisfy  (2) then  the procedure could continue with $c'_3$ and so on. However, after a while we must get a path $a_1,\ldots,a_{p_{23}}$ such that   $T_{a_{p_{23}}}\cap \{a_1,\ldots,a_{p_{23}-1}\}=\emptyset$,  no  bad path starts with $a_{p_{23}}$ and for every $i=2,3$ one of the following statements holds:

$j_{23}) $ $U_{a_{p_{23}}}\cap (u_i)=\emptyset$,

 $j_{23}')$
   $U_{a_{p_{23}}}\cap (u_i)\not =\emptyset$  there exist   $b_i\in T_{a_{p_{23}}}\cap (f_i)$ with $h(b_i)\in (u_i)$,

$j_{23}'')$ $U_{a_{p_{23}}}\cap (u_i)\not =\emptyset$ and every monomial $m\in U_{a_{p_{23}}}\cap (u_i)$ has
 all its divisors from $B\cap (f_i)$  contained in  $\{u_2,u'_2,\ldots,u_4,u'_4\}$.

 We end  the proof applying  the same procedure  with $c'_4$ together  with $c'_2$, $c'_3$ and if necessary Lemma \ref{i0}.
\hfill\ \end{proof}

\begin{Remark}\label{v'} {\em  Using the properties (2), (3) mentioned  in Lemma \ref{i0} we may have  $b_i=w_{1i}$, for some $2\leq i\leq 4$  only if $u_i,u'_i\in W$. Thus, let us say  $b_2=w_{12}$  only if $\{u_2,u'_2\}=\{w_{23},w_{24}\}$.  Then $\{u_i,u'_i\}\not \subset W$ for $i=3,4$ and so $b_3\not = w_{13}$,  $b_4\not = w_{14}$,
in case $b_3,b_4$ are given by Lemma \ref{v}. Therefore at most one from $b_i$ could be  $w_{1i}$.}
\end{Remark}

The idea of the proof of Proposition \ref{ml}
fails in a special case  hinted by Example \ref{b'}. This case is solved directly by the following lemma.

\begin{Lemma}\label{vi} Suppose that   $b=x_jf_1$ and  $(B\setminus E)\subset W\cup \{x_jf_1,x_jf_2,x_jf_3,x_jf_4\}$ for some $j\not\in  \supp f_1$. Then $\depth_SI/J\leq d+1$.
\end{Lemma}
\begin{proof}
 If   $|B\setminus E|<2r=8$ then $\depth_SI''/J''\leq 2 $ by \cite[Theorem 2.4]{Sh}. Assume that $|B\setminus E|\geq 8$. Our hypothesis gives $|B\cap W|\geq 4$.
First assume that $5\leq |B\cap W|\leq 6$ and we get that let us say $f_i=vx_i$, $ 1\leq i\leq 4$ for some monomial $v$ of degree $d-1$ (see the proof of \cite[Lemma 3.2]{PZ2}). Then
$$\depth_SI/J=\deg v+\depth_{S'}((I:v)\cap S')/((J:v)\cap S'),$$
 $S'=K[\{x_i:i\in ([n]\setminus \supp v)\}]$ and it is enough to show the case $v=1$, that is $d=1$.

We may assume that $f_i=x_i$, $i\in [4]$ and $j=5$ since $b\not\in W$. It follows that  $(B\setminus E)\subset W\cup \{b,x_2x_5,x_3x_5,x_4x_5\}$.
Set $I''=(x_1,\ldots,x_4)$, $J''=J\cap I''$.
  Note that $J\supset (x_1,\ldots,x_5)(x_6,\ldots,x_n)$ and so $\depth_SI''/J''=\depth_{S''} (I''\cap S'')/(J''\cap S'')$ for $S''=K[x_1,\ldots,x_5]$.

 Then $J''\cap S''$ is generated by at most two monomials and so $\depth_{S''}S''/(J''\cap S'')\geq 3$. Since $\depth_{S''}S''/(I''\cap S'')=1$ it follows that $\depth_SI''/J''=\depth_{S''} (I''\cap S'')/(J''\cap S'')=2$.
Therefore  $\depth_SI/J\leq 2$ either when $E=\emptyset$ or by the Depth Lemma since $I/(J,I'')$ is generated by monomials of $E$ which have degrees $2$.

Now assume that $|B\cap W|=4$, let us say  $B\cap W=\{w_{14},w_{23},w_{24},w_{34}\}$. Then we may suppose that  $f_i=vx_ix_6$, $2\leq i\leq 4$ and $f_1=vx_1x_4$ for some monomial $v$ of degree $d-2$. As above we may assume that $v=1$ and $n=6$.
If $j=6$ then $b=w_{14}$ which is impossible. If let us say $j=2$ then $(B\setminus E)\subset W\cup \{b,x_2x_3x_6,x_2x_4x_6\}$ and so $|B\setminus E|<8$, which is false.

Thus $j\not \in\{1,\ldots,4,6\}$ and we may assume that $j=5$.
It follows that $J\subset (x_1x_2x_6,x_1x_3x_6,x_1x_2x_4,x_1x_3x_4)$, the inclusion being strict only if $|B\setminus E|<8$ which is not the case. Thus  $J= (x_1x_2x_6,x_1x_3x_6,x_1x_2x_4,x_1x_3x_4)$ and a computation with SINGULAR shows that $\depth_SI/J=3$ in this case.
\hfill\ \end{proof}

 Next we put together the above lemmas to get the proof of Proposition \ref{ml}. Assume that $\sdepth_SI/J\leq d+1$.
 We may suppose always that $P_b$  satisfies the properties mentioned in Lemma \ref{i0}.  Applying Lemma \ref{ii2} and Remark \ref{ii'} and changing $a_1$ if necessary we may suppose that no bad path starts from $a_1$.  By  Lemma \ref{v} changing $a_1$ by $a_p$ we may suppose that for every $i=2,3,4$ one of the following statements holds

1) $U_{1}\cap (u_i)=\emptyset$,

2) $U_{1}\cap (u_i)\not =\emptyset$
and  there exists $b_i\in T_1\cap (f_i)$ with $h(b_i)\in (u_i)$,

3)  $U_{1}\cap (u_i)\not =\emptyset$ and every monomial of  $U_1\cap (u_i)$ has
 all its divisors from $B\cap (f_i)$  contained in  $\{u_2,u'_2,\ldots,u_4,u'_4\}$.

Mainly we study  case 3) the other two cases are easier  as we will see later. Suppose that  $U_{1}\cap (u_2)\not =\emptyset$ and every monomial of  $U_{1}\cap (u_2)$ has
 all its divisors from $B\cap (f_2)$  contained in  $\{u_2,u'_2,\ldots,u_4,u'_4\}$. Let $m\in U_{1}\cap (u_2)$, let us say $m=h(a_e)$ for some path $a_1,\ldots,a_e$.
 be as in case 3). We may suppose that $U_1\cap (u'_2)=\emptyset$ because otherwise we may assume as in Lemma \ref{iv} that all divisors of $c'_2$ are in the enlarged $T'_1$ of $T_1$ and so $c'_2$ is preserved.  As in the proof of Lemma \ref{iv''} one of the following statements holds:

  $1')$ $U_1\cap (u_2)= \{m\}$, $m\in (u_2)\cap (u_3)$, $u_3=w_{23}$, $m\not\in (u_4,u'_4)$ and there exists ${\tilde a}_3\in T_1\cap (f_3)$ dividing $m$ with ${\tilde a}_3=a_e$,

 $2')$ $U_1\cap (u_2)= \{m\}$, $m\in (u_2)\cap (u_3)$, $u_3=w_{23}$, $m\not\in (u_4,u'_4)$ and there exists ${\tilde a}_3\in T_1\cap (f_3)$ dividing $m$  with ${\tilde a}_3\not=a_e$,

 $3')$ $U_1\cap (u_2)= \{m\}$,  $m\in (u_2)\cap (u_4)$, $u_4=w_{24}$, $m\not\in (u_3,u'_3)$ and there exists ${\tilde a}_4\in T_1\cap (f_4)$ dividing $m$  with ${\tilde a}_4=a_e$,

 $4')$ $U_1\cap (u_2)= \{m\}$,  $m\in (u_2)\cap (u_4)$, $u_4=w_{24}$, $m\not\in (u_3,u'_3)$ and there exists ${\tilde a}_4\in T_1\cap (f_4)$ dividing $m$ with ${\tilde a}_4\not =a_e$,

  $5')$  $m=w_{24}\in (u_2)\cap (u_3)\cap (u_4)$, $u_3=w_{23}$, $u_4=w_{34}$ and there exists ${\tilde a}_4\in T_1\cap (f_4)$ dividing $m$  with $h({\tilde a}_4)=m$,

$6')$  $m=w_{24}\in (u_2)\cap (u_3)\cap (u_4)$, $u_3=w_{23}$, $u_4=w_{34}$ and there exists ${\tilde a}_4\in T_1\cap (f_4)$ dividing $m$  with $h({\tilde a}_4)\not=m$,

   $7')$  $m=\omega_1\in C_3$, $u_2=w_{24}$, $u_3=w_{23}$.

In  subcase $1')$  change in $P_b$ the intervals $[f_3,c'_3]$, $[{\tilde a}_3,m]$ with $[f_3,m]$, $[u'_3,c'_3]$. The new $T_1''=T_1\setminus \{{\tilde a}_3\}$ corresponds to $U_1''=U_1\setminus \{m\}$  which has empty intersection with $(u_2)$ by our assumption. If $T_1''$ is not empty then we may go on with $T_1''$ instead $T_1$, the advantage being that now we have no problem with $u_2$. If $T_1''=\emptyset$ then $e=1$ and the path $a_1$ is maximal. Since $m\not \in(u_4,u'_4)$ we must have $u_2=x_kf_2$ for some $k$ (we can also have  $w_{12}=x_kf_2$) and so $m=x_kw_{23}$, ${\tilde a}_3=x_kf_3$. If $E\not=\emptyset$ then we may change $a_1$ by a monomial of $E$. Assume that  $E=\emptyset$. If $c'_3=x_tw_{23}$ for some $t$ then $x_tf_2\in B$ since it divides $c'_3$. If $t=k$ then $m=c'_3$. Thus $t\not= k$, $x_tf_2\not\in \{b,u_2,u'_2,\ldots,u_4,u'_4\}$ and we may change   $a_1$ by $x_tf_2$ and the new $T_1''$ will be not empty.  If $c'_3\in C_2$ we may  find also a divisor $b'\in B\setminus  \{b,u_2,u'_2,\ldots,u_4,u'_4\}$ dividing $c'_3$ and changing $a_1$ by  $b'$ we will get the new $T_1''$ not empty.  Remains to assume that $c'_3\in C_3$. Then $u'_3=w_{34}$ and $b''=w_{24}$ is either in $\{u_2',u_4,u_4'\}$, or we may change $a_1$ by $b''$ as above. Suppose that $u'_2=w_{24}$. Then $x_kf_4\in B$. If $x_kf_4\not\in  \{b,u_2,u'_2,\ldots,u_4,u'_4\}$  we may change   $a_1$ by $x_kf_4$. Otherwise, let us say $u_4=x_kf_4$ and  $c'_4=x_kw_{14}$.  We get $x_kf_1\in B\setminus  \{u_2,u'_2,\ldots,u_4,u'_4\}$ and if $b\not = x_kf_1$ then we may change as above $a_1$ by $x_k f_1$. If $b=x_k f_1$ then note that $B\supset \{w_{23},w_{24}.w_{34},w_{14},b,x_kf_2,x_kf_3,x_kf_4\}$. If there exists a monomial $b'\in B\setminus (W\cup \{b,x_kf_2,x_kf_3,x_kf_4\})$ then change $a_1$ by $b'$. Otherwise $B\subset W\cup \{b,x_kf_2,x_kf_3,x_kf_4\}$ and we apply Lemma \ref{vi}.

Therefore in this subcase changing $P_b$ ($u_3$ is preserved and the new $u'_3$ is $b_3$) and passing from $T_1$ to $T_1''$ there exist no problem with $u_2$. As in Lemma \ref{iv} we may suppose that only one from $U_1''\cap (u_3)$,  $U_1''\cap (u'_3)$ is nonempty because otherwise we preserve the new $c'_3$, that is $m$. If let us say $U_1''\cap (u_3)=\{m'\}$, and all divisors of $m'$ from $B\cap (f_3)$ are contained  in $\{u_3,u'_3,u_4,u'_4\} $ then $m'\in (u_3)\cap (u_4)$, $u_4=w_{34}$ and there exists ${\tilde a}_4\in T_1''\cap (f_4)$ dividing $m'$. If $h({\tilde a}_4)=m'$ then as above change in $P_b$ the intervals $[f_4,c'_4]$, $[{\tilde a}_4,m']$ with $[f_4,m']$, $[u_4',c'_4] $. Clearly ${\tilde T}_1=T_1''\setminus \{{\tilde a}_4\}$ has empty intersection with $(u_3)$ and similarly to above we may suppose that ${\tilde T}_1\not =\emptyset$.
 In this way we arrive to the situation when we will not meet case 3) for $2\leq i\leq 4$.

In subcase $2')$ we have $a_e\in E$ and  $a_{e+1}={\tilde a}_3\in T_1$. Take $T_{a_{e+1}}$ instead $T_1$. If $a_e$ will not appear anymore in $T_{a_{e+1}}$ then $U_{a_{e+1}}\cap (u_2)=\emptyset$ and the problem is solved. Otherwise, if $a_v=a_{e}$ for some $v>e+1$ then change in $P_b$ the intervals $[a_i,h(a_i)]$, $e\leq i\leq v$ with $[a_{i+1},h(a_i)]$, $e\leq i<v$, $[a_e,m_v]$ we see that the new $a_{e}$ is the old $a_{e+1}$, that is we reduced to the subcase $1')$. Subcases $3')$, $4')$ are similar to $1')$, $2')$.

 Change in subcase $5')$ (as in subcase $1')$) the intervals $[f_4,c'_4]$, $[{\tilde a}_4,m]$ of  $P_b$  with $[f_4,m]$, $[u'_4,c'_4]$. The new $T_1''=T_1\setminus \{{\tilde a}_4\}$ corresponds to $U_1''=U_1\setminus \{m\}$  which has empty intersection with $(u_2)$ by our assumption. The proof continues as in $1')$. Similarly, $6')$ goes as  $2')$.

 In subcase $7')$ if $\omega_1\in W$ (see Example \ref{exe}) then  it has  $4$ divisors from $B\setminus E$  and so one of them is not in $\{u_2,u'_2,\ldots,u_4,u'_4\}$  and we may proceed as in subcases $5')$, $6')$. So we may assume that $\omega_1\not \in W$. Then  either $u_4=w_{34}$ and then $a_e\in E$ which is false by our assumption, or $w_{34}\in T_1$. Set $a_{e+1}=w_{34}$.  We proceed as in $2')$ taking $T_{a_{e+1}}$ if $a_e\not \in T_{a_{e+1}}$ or otherwise changing $P_b$ we reduce to the situation when $h(a_{e+1})=m$. Then change in $P_b$ the intervals $[f_4,c'_4]$, $[a_{e+1},m]$ with $[f_4,m]$, $[u'_4,c'_4]$ and as usual the new    $U_1''=U_1\setminus \{m\}$   has empty intersection with $(u_2)$.

Thus we may assume that for all $2\leq i\leq 4$ we are  in cases 1), 2). When we are in case 2) there exists $b_i\in T_1\cap (f_i)$ with $h(b_i)\in (u_i)$ and we may consider the  intervals
$[f_i,c'_i]$, which are disjoint since $b_i$ are different by Lemma \ref{v}.  Moreover, they  contain at most one monomial from  $w_{12}, w_{13},w_{14}$ by Remark \ref{v'}, which is useful next.
Remains to study  those $i$ with  $ U_1\cap (f_i)\not =\emptyset$ but $U_1\cap (u_i,u'_i)=\emptyset$. If $U_1\cap (u_2,u'_2,\ldots,u_4,u'_4)=\emptyset$ then we apply Lemma \ref{i}. Suppose that $ U_1\cap (f_2)\not =\emptyset$ and $U_1\cap (u_2,u'_2)=\emptyset$ but we found already $b_3$ and possible $b_4$ as in 2). If $h(b_3)\not \in (f_2)$ then choosing $b'\in B\cap (f_2)$ we see that the intervals $[f_2,h(b')]$, $[f_3,h(b_3)]$ are disjoint. A similar result holds if there exists $b_4$ and $h(b_4)\not \in (f_2)$.

Assume that $h(b_3) \in (f_2)$. Then we may suppose that $u_3=w_{23}$ and $h(b_3)=x_kw_{23}$ for some $k\in [n]\setminus \supp w_{23}$. We claim  that $b''=x_kf_2\not\in \{u_2,u'_2,\ldots,u_4,u'_4\} $.
It is clear that $b''\not\in \{u_2,u'_2,u_3,u'_3\} $. If $b''\in \{u_4,u'_4\}$ then $b''=w_{24}=u_4$, let us say. Thus $h(b_3)\in (u_3,u_4)$ but $h(b_3)\not\in (u_2,u'_2)$. This means that the monomial $h(b_3)\in U_1\cap (u_4)$ is in the situation 3) (similarly to $1')$) which is not possible as we assumed. This shows our claim.

Therefore, $b''\in T_1\cap (f_2)$ because it divides $h(b_3)$. If $h(b'')\in (f_3)$ then $h(b'')=kw_{23}=h(b_3)$ which is impossible.
If $h(b'')\in (f_4)$ then $h(b'')=x_tw_{24}$ for some $t$. As we saw above $b''\not = w_{24}$ and so $t=k$. If $b_4$ is not done by 2) then it is enough to note that the intervals $[f_2,h(b'')]$, $[f_3,h(b_3)]$ are disjoint. Assume that $b_4$ is given already from 2) and $u_4=w_{24}$. Then ${\tilde b}=x_kf_4\not = u'_4$ because otherwise  $h(b'')=h(b_4)$. We see that ${\tilde b}\not \in \{u_2,u'_2,\ldots,u_4,u'_4\} $ and so $\tilde b$ is in $T_1\cap (f_4)$. But $h({\tilde b})\not\in (u_4)$ because it is different of $h(b_4)$. Then the intervals $[f_2,h(b'')]$, $[b_3,h(b_3)]$, $[f_4,h({\tilde b})]$ are disjoint.
As in Lemma \ref{i} we find if necessary an interval $[f_1,c]$ disjoint of the rest.

  Suppose as in Lemma \ref{i} that $[r]\setminus \{j\in [r]: U_1\cap (f_j)\not=\emptyset\}=\{k_1,\ldots,k_{\nu}\}$ for some  $1\leq k_1<\ldots <k_{\nu}\leq 4$, $0\leq \nu\leq 4$. Set  $I' = (f_{k_1},\ldots,f_{k_{\nu}},G_1)$,  $J'= I' \cap J$,
With the help of the above disjoint intervals, $P_b$ induces on  $I/(I',J)$  a partition $P'_b$ with sdepth $d+2$. It follows that $\sdepth_SI'/J'\leq d+1$ using \cite[Lemma 2.2]{R}.   By Lemma \ref{dep} we get $\depth_SI/(J,I')\leq d+1$ and we are done.
$\hfill\ \square$

\begin{Remark} \label{M'} {\em Note that in $P'_b$, all divisors from $B$ of the new $c'_i$ are in $T_1\cup \{u_2,u'_2,\ldots,u_4,u'_4\}$. If one old $c'_i$ has already this property then we may keep it.}
\end{Remark}

\begin{Remark}\label{M''}{\em If $\omega_1\in (C_3\setminus W)\cap (E)$ then we may have indeed a problem. For example, if $u_2=w_{24}$, $u_3=w_{23}$, $u_4=w_{34}$, $\omega_1=h(a_1)$ for some $a_1\in E$ but $\omega_1\not \in h(E\setminus \{a_1\})$ then the path $a_1$ is maximal, $T_1=\{a_1\}$ and our theory fails to solve this case if we cannot change  $P_b$ in order to have    $\{u_2,u_3,u_4\}\not =\{w_{24},w_{23},w_{34}\}$.}
\end{Remark}
\begin{Example} \label{bad} {\em
We continue Example \ref{b2}.  If we take as in the above proof $I'=(b,x_5x_6,x_5x_7)$ and $J'=I'\cap J$ we have  the disjoint intervals $[x_i,c'_i]$, $2\leq i\leq 4$ and to conclude that $h $ induces a partition on $I/(I',J)$, which has sdepth $3$ we need an interval $[x_1,c'_1]$ disjoint of the other ones. But this is hard because there are too many $w_{1i}$ among $\{u_2,u'_2,\ldots,u_4,u'_4\}$. We must change one $c'_i$ with one $m\in (U_1\cap (x_i))\setminus (x_1)$. The only possibility is to take $m_2=x_2x_3x_5$. Since $m\in (u'_2)\setminus (u_3,u_3',u_4,u'_4)$  we may change somehow $c'_2$ with $m$. This is not easy since $m_2=h(a_2)$, $a_2=x_3x_5\not\in (x_2)$. As in Lemma \ref{iv} note that $a_1|m_3=h(a_3)$ and replacing in $P_b$ the intervals $[a_i,m_i]$, $i\in [3]$, $m_1=h(a_1)$ with the intervals $[a_1,m_3]$, $[a_2,m_1]$, $[a_3,m_2]$ we see that  $x_2x_5$ - the new $a_2$, belongs to  $(x_2)$. Thus  we may change in $P_b$ the intervals $[x_2,c'_2]$, $[x_2x_5,m_2]$ with $[x_2,m_2]$, $[u_2,c'_2]$. The new $T_1$ is $T'_1=(T_1\cup \{x_1x_2\})\setminus \{x_2x_5\}$. Note that all divisors from $B\cap (x_2)$ of the new $c'_2$ which are different from the new $u_2,u'_2$ are contained in the new $T_1$. As above $[x_i,c'_i]$ are disjoint intervals and changing in $P_b$ the intervals $[x_1x_2,x_1x_2x_3]$, $[x_1x_5,x_1x_2x_5]$ with $[x_1,x_1x_2x_5]$ we get a partition with sdepth $3$ on $I/(I',J)$.
}
\end{Example}

\section{Main results}

We start with an elementary lemma closed to Lemma \ref{vi}.
\begin{Lemma} \label{ele}
Let $r$ be arbitrarily chosen, $r'\leq r$, $t\in [n]\setminus \cup_{i=1}^{r'}\supp f_i$ and $I'=(f_1,\ldots,f_{r'})$, $J'=J\cap I'$. Suppose that all $w_{ij}$, $1\leq i<j\leq r'$ are in $B$ and different. Then the following statements hold
\begin{enumerate}
\item{} there exists a monomial $v$ of degree $d-1$ such that $f_i\in (v)$ for all $i\in [r']$,
\item{} if $x_k(f_1,\ldots, f_{r')}\subset J$ for all $k\in [n]\setminus (\{t\}\cup(\cup_{i=1}^{r'}\supp f_i))$ then $\depth_SI'/J'\leq d+1$.
\end{enumerate}
\end{Lemma}
\begin{proof}
As in the proof of \cite[Lemma 3.2]{PZ2} we may suppose that  $f_i=vx_i$ for $i\in [r]$ and some monomial $v$ of degree $d-1$, that is (1) holds.
 It follows that
 $$\depth_SI'/J' =d-1+\depth_{S''}(x_1,\ldots,x_{r'})S''= d+1$$ where  $S''=K[x_1,\ldots,x_{r'},x_t]$.
 \hfill\ \end{proof}

\begin{Theorem} \label{m1} Conjecture \ref{c} holds for $r\leq 4$, the case $r\leq 3$ being given in Theorem \ref{ap}.

\end{Theorem}
\begin{proof} Suppose that $\sdepth_SI/J=d+1$ and $E\not =\emptyset$, the case $E=\emptyset$ is given in Proposition \ref{m'}. The proofs of Proposition \ref{ml} and Proposition \ref{m'} show that we get $\depth_SI/J\leq d+1$, that is Conjecture \ref{c} holds, when we may choose $b_i\in (B\cap (f_i))\setminus W$ such that $\omega_i\not\in (C_3\setminus W)\cap (E)$. Suppose that we choose $b_1 \in (B\cap (f_1))\setminus W$ but $\omega_1\in (C_3\setminus W)\cap (E)$. In the last part of the proof of Proposition \ref{ml} (see $7')$ and also Remark \ref{M''})  a problem appears when $m=\omega_1\in T_1$ and let us say $u_2=w_{24}$, $u_3=w_{23}$, $u_4=w_{34}$. As in the proof of \cite[Lemma 3.2]{PZ@}  we
may assume that $f_i=vx_i$ for $2\leq i\leq 4$ and  some monomial $v$ of degree $d-1$. If let us say $x_tf_2\in B$ for some $t\not\in \cup_{i=2}^4\supp f_i$ then either $tf_2=w_{12}$, or $tf_2\not \in W$. In the first case we may suppose, as in the proof of Lemma \ref{vi}, that
 one of the following statements hold:

1) $f_i=vx_i$, $i\in [4]$ for some monomial $v$ of degree $d-1$,

2) $f_i=px_ix_5$, $2\leq i\leq  4$, $f_1=px_1x_2$ for some monomial $p$ of degree $d-2$.

In both cases we see that if $B\cap (f_2,f_3,f_4)\subset W$ then we have $x_k(f_2,\ldots f_{4})\subset J$ for all $k\in [n]\setminus (\{1\}\cup(\cup_{i=2}^{4}\supp f_i))$.
 By  Lemma \ref{ele}  we get $\depth_SI'/J'\leq d+1$ for $I'=(f_2,f_3,f_4)$, $J'=J\cap I'$  which gives $\depth_SI/J\leq d+1$ since $\depth_SI/(J,I')\geq d+1$, $b$ being not in $(J,I')$. Thus $B\cap (f_2,f_3,f_4)\not \subset W$ and we may choose, let us say $b_2 \in (B\cap (f_2))\setminus W$ and again we may get $\depth_SI/J\leq d+1$ if $\omega_2\not\in (C_3\setminus W)\cap (E)$.

Thus we may assume that $\omega_1,\omega_2\in (C_3\setminus W)\cap (E)$. In particular $B\cap W$ consists in at least $5$ different monomials and so we may suppose that 1) above holds and $u'_2=vx_2x_{k_2}$, $u'_3=vx_3x_{k_3}$, $u'_4=vx_4x_{k_4}$ for some $k_i\in ([n]\setminus (\{2,3,4\}\cup \supp v) $.
 If $k_2=k_3=k_4=1$ then $c'_2=\omega_3$, $c'_3=\omega_4$, $c'_4=\omega_2$, that is all $\omega_i$ are in $C_3\setminus W$. If let us say $k_3>4$ then $b''=x_{k_3}f_3\not\in W$ and we are ready if $\omega_3\not \in (C_3\setminus W)\cap (E)$. Thus we may assume that $\omega_3 \in (C_3\setminus W)\cap (E)$. Consequently in all cases we may assume that $3$ from $\omega_i$ are in $C_3\setminus W$. In particular $|B\cap W|=6$. If $B\cap (f_i)\subset W$ for some $i=3,4$ then $(J:f_i)$ is generated by $x_j$ with $j\not \in  (\{1,\ldots,4\}\cup \supp v)$. It follows that in the exact sequence
$$0\to (f_i)/J\cap (f_i)\to I/J\to I/(J,f_i)\to 0$$
the first term has depth $\deg v+4=d+3$ and sdepth $\geq d+2$. By \cite[Lemma 2.2]{R}
we get $\sdepth_SI/(J,f_i)\leq d+1$ and so the last term in the above sequence  has depth $\leq d+1$ by Theorem \ref{ap}. Using the Depth Lemma we get $\depth_SI/J\leq d+1$ too.

Therefore, we may find $b_i\in (B\cap (f_i))\setminus W$, $i=3,4$ and as above we may suppose that $\omega_i\in (C_3\setminus W)\cap (E)$, let us say $\omega_i\in ({\tilde a}_i)$ for some ${\tilde a}_i\in E$. We consider three cases depending on $k_i$.

{\bf Case 1}, when $k_i=1$ and $k_j>4$ for  some $i,j=2,3,4$, $i\not=j$.

Assume that $k_2=1$, that is $c'_2=\omega_3$ and $k_4>1$. Then $a_1=vx_1x_4\not \in \{u_2,u'_2,\ldots,u_4,u'_4\}$  is a divisor of $c'_2$. Start the usual proof with $a_1$ and if  $\omega_1\not\in U_1$ then we get $\depth_SI/J\leq d+1$.  Suppose that there exists a (possible bad) path $a_1,\ldots,a_e$, $m_i=h(a_i)$ such that $m_e=\omega_1$. Changing in $P_b$ the intervals $[a_i,m_i]$, $i\in [e]$, $[f_2,c'_2]$, $[f_3,c'_3]$ with $[a_{i+1},m_i]$, $i\in [e-1]$, $[f_1,c'_2]$, $[f_2,m_e]$, $[u'_3,c'_3]$ we see that the new ${\tilde c}'_i$, $i=1,2,4$ contain two from $\omega_i$. Choose a new $a_1$ and start to build $U_1$. This time any monomial from $U_1$ has at least one divisor from $B\setminus E$ which is not in $\cup_{j=1,2,4} [f_j,{\tilde c}'_j]$ so the usual proof goes.

{\bf Case 2}, $k_2,k_3,k_4>4$.

 Then $a_1=vx_1x_4\not \in \{u_2,u'_2,\ldots,u_4,u'_4\}$. Let $m_1=h(a_1)=a_1x_k$ for some $k$. If $k=k_4$ then changing in $P_b$ the intervals $[f_4,c'_4]$,
 $[a_1,m_1]$ with $[f_4,m_1]$, $[u_4,c'_4]$ we see that $u_4=w_{34}$ does not divide the new $c'_4$ and so we have no problem with $\omega_1$.

 Suppose that  $k\not =k_4$ and $k>4$ then $a_2=vx_4x_k\not \in \{u_2,u'_2,\ldots,u_4,u'_4\}$.
   If there exists no path $a_2,\ldots,a_e$, $m_i=h(a_i)$ with $m_e=\omega_1$ then we proceed as usual.
  Otherwise, let $a_2,\ldots,a_e$, $m_i=h(a_i)$ be a (possible bad) path with $m_e=\omega_1$. Changing in $P_b$ the intervals  $[a_i,m_i]$, $i\in [e]$, $[f_3,c'_3]$, $[f_4,c'_4]$ with $[a_{i+2},m_{i+1}]$, $i\in [e-2]$,  $[f_3,m_e]$, $[f_4,m_1]$, $[u'_3,c'_3]$, $[u'_4,c'_4]$ we see that any monomial from $C$ has at least one divisor from $B\setminus E$ which is not in $\cup_{j=2,3,4} [f_j,{\tilde c}'_j]$ so the usual proof goes, where ${\tilde c}'_j$ denotes the new $c'_j$ for $j=3,4$ and ${\tilde c}'_2=c'_2$.

 Remains to study the case when $k\not =k_4$ and $k=2$ or $k=3$. Assume that $k=2$, that is  $m_1=\omega_3$. Similarly we may assume that $a_2=vx_1x_2$, $m_2=h(a_2)=a_2x_{3}=\omega_4$ and $a_3=vx_1x_3$, $m_3=h(a_3)=a_3x_{4}=\omega_2$. If there exists no path $a_3,\ldots,a_e$, $m_i=h(a_i)$ with $m_e=\omega_1$ then we proceed as usual. Otherwise,  let  $a_3,\ldots,a_e$, $m_i=h(a_i)$ be a (possible bad) path with $m_e=\omega_1$.  Changing in $P_b$ the intervals  $[a_i,m_i]$, $i\in [e]$, $[f_j,c'_j]$, $j=2,3,4$ with $[a_{i+3},m_{i+2}]$, $i\in [e-3]$, $[f_1,m_1]$,  $[f_3,m_2]$, $[f_4,\omega_1]$, $[u'_2,c'_2]$, $[u'_3,c'_3]$, $[u'_4,c'_4]$  we arrive in a case similar to the next one.

{\bf Case 3},  $k_2=k_3=k_4=1$.

Thus $c'_2=\omega_3\in (a_1)$ for  $a_1={\tilde a}_3$. If there exists a path $a_1,\ldots,a_e$, $m_i=h(a_i)$ with $m_e=\omega_1$ then changing in $P_b$ the intervals  $[a_i,m_i]$, $i\in [e]$, $[f_2,c'_2]$, $[f_3,c'_3]$ with   $[a_{i+1},m_i]$, $i\in [e-1]$, $[a_1,c'_2]$,  $[f_1, c'_3]$, $[f_2,\omega_1]$ we get  the new ${\tilde c}'_1=\omega_4$, ${\tilde c}'_2=\omega_1$ and ${\tilde c}'_4=c'_4=\omega_2$. Thus we may change the three $c'_i$ to be any three monomials from  $\omega_j$.

Assume that the above path is bad, let us say $m_p\in (b)$ for $p<e$ and as in Lemma  \ref{ii2} we may suppose that $a_{p+1}\not\in E$,   $T_{a_{p+1}}\cap \{a_1,\ldots,a_p\}=\emptyset$  and there exists no bad path starting with $a_{p+1}$. Changing $P_b$ as above we see that  the new ${\tilde c}_i'$ are $\omega_1,\omega_2,\omega_4$ and the $\omega_3\not\in U'_{a_{p+1}}$, where $U'_{a_{p+1}}$ corresponds to $T'_{a_{p+1}}=T_{a_{p+1}}\setminus \{a_{p+1}\}$. Set $b'=a_{p+1}$. In fact changing in the new $P_b$ the intervals $[b',m_p]$ with $[b,m_p]$ we get a partition $P_{b'}$ on $I_{b'}/J_{b'}$, where $I_{b'}J_{b'}$  are defined as usually but we could have $b'\in W$.
 There exists no bad path in $P_{b'}$ because otherwise this induces one in $P_b$.
 We may proceed as before since all monomials from $U'_{b'}$ has   at least one divisor from $B\setminus E$ which is not in $\cup_{j=1,2,4} [f_j,{\tilde c}'_j]$.
Similarly, we do for any $a_1\in E$ dividing one from $c'_2,c'_3,c'_4$ and remains to assume that there exists  no bad path starting with a divisor from $E$ of any $c'_i$, $i=2,3,4$.

Now suppose that $a_1=b_3$ and consider $T_1,U_1$ as usual and we may suppose that we are still in Case 3 but with $(\tilde c'_j)$, $j=1,3,4$.
If there exists no bad path starting with $a_1$ and
$m_1=h(a_1)\in (W)$, let us say $m_1\in (w_{13})$ then changing in $P_b$ the intervals $[a_1,m_1]$, $[f_1,\tilde c'_1]$ with $[f_1, m_1]$, $[\tilde u_1,\tilde c'_1]$, $\tilde u_1=w_{12}$ we arrive in a case similar to Case 1. If $m_1\not\in (W)$ then
 assume that in $P_b$ there exist the intervals $[f_1,\omega_2]$, $[f_2,\omega_4]$, $[f_4,\omega_1]$. Then
$[f_3,m_1]$ is  disjoint of these intervals. Enlarge $T_1$ to $\tilde T_1$ adding all monomials from $B$ connected by a path which is not bad, with
the divisors from $E$ of $(\omega_j)$, $j=1,2,4$. Thus taking $I'=(B\setminus
(\tilde T_1\cup W))$, $J'=J\cap I'$ we get $\sdepth_SI/(J,I')\geq d+2$ which is enough as usual.

If there exists a bad path $a_1,\ldots, a_e$, $m_i=h(a_i)$, $m_e=\omega_1$, $m_p\in (b)$, $p<e$ then as above we may assume that  $a_{p+1}\not\in E$,   $T_{a_{p+1}}\cap \{a_1,\ldots,a_p\}=\emptyset$  and there exists no bad path starting with $a_{p+1}$. Moreover, we may  choose  $a_{p+2}\not\in E$ when $e>p+1$
because $m_{p+1}\not=\omega_1$. Taking as above $b'=a_{p+1}$ and the partition $P_{b'}$ given on $I_{b'}/J_{b'}$ we see that $T_{a_{p+2}}\cap (f_1,\ldots,f_4)\not =\emptyset$ and we reduce to the above situation with $T_{a_{p+2}}$ instead $T_1$. If  $p\geq e-1$ then $\omega_1\not \in U_{a_{p+2}}$ and   so there exists no problem.
  \hfill\ \end{proof}

\begin{Theorem} \label{m2} Conjecture \ref{c}  holds for $r=5$ if
there exists $t\in [n]$ such that $t\not\in \cup_{i\in [5]} \supp f_i$, $(B\setminus E) \cap (x_t)\not =\emptyset$ and $E\subset (x_t)$.
\end{Theorem}
\begin{proof} Apply Lemma \ref{el}, since Conjecture \ref{c} holds for $r\leq 4$ by
 Theorem \ref{m1}.
\hfill\ \end{proof}

\begin{Example} {\em
Let $n=8$, $E=\{x_6x_7,x_7x_8\}$, $I=(x_1,x_2,x_3,x_4,x_5,E),$

$ J=(x_1x_6,x_1x_8,x_2x_8,x_3x_6,x_3x_8,x_4x_6,x_4x_7,x_4x_8,x_5x_6,x_5x_7,x_5x_8)$.

\noindent We see that we have
$B=$
$$\{x_1x_2,x_1x_3,x_1x_4,x_1x_5,x_1x_7,x_2x_3,x_2x_4,x_2x_5,x_2x_6,x_2x_7,x_3x_4,x_3x_5,x_3x_7,x_4x_5\}\cup \{E\},$$
$$C=\{x_1x_2x_3,x_1x_2x_4,x_1x_2x_5,x_1x_2x_7,x_1x_3x_4,x_1x_3x_5,x_1x_3x_7,x_1x_4x_5,x_2x_3x_4,x_2x_3x_5,$$
$$x_2x_3x_7,x_2x_4x_5,x_2x_6x_7,x_3x_4x_5,x_6x_7x_8\}$$ and so
$ r=5$,  $q=15$, $s=16 \leq q+r$.
We have $\sdepth_S I/J = 2$, because otherwise the monomial $x_2x_6$ could enter  either in $ [x_2,x_2x_6x_7]$, or in $ [x_2x_6,x_2x_6x_7]$ and in both cases remain the  monomials of $E$ to enter in an interval ending with $x_6x_7x_8$, which is impossible.
Then $\depth_SI/J\leq 2$ by the above theorem since $E\subset (x_7)$ and for instance $x_1x_7\in (B\setminus E)\cap (x_7)$.}
\end{Example}

\address{Dorin Popescu,  Simion Stoilow Institute of Mathematics of Romanian Academy, Research unit 5,
 P.O.Box 1-764, Bucharest 014700, Romania \\
\email{dorin.popescu@imar.ro}

 \end{document}